\documentclass[11pt]{amsart}
\usepackage{amsthm}
\usepackage{amsmath}
\usepackage{amssymb}
\usepackage{latexsym}
\usepackage{bm}
\usepackage{pstricks}
\usepackage{pst-plot}
\usepackage{psfrag}
\usepackage{graphicx}
\usepackage{ setspace}
\usepackage[all]{xy}
\usepackage[margin=1.25in]{geometry}
\theoremstyle{plain}
\newtheorem{thrm}{Theorem}[subsection]
\newtheorem{prop}[thrm]{Proposition}
\newtheorem{lem}[thrm]{Lemma}

\theoremstyle{definition}
\newtheorem{defn}[thrm]{Definition}
\theoremstyle{remark}
\newtheorem{remark}[thrm]{Remark}
\newtheorem{example}[thrm]{Example}

\newcommand\M{G}
\newcommand\N{\mathbb{N}}
\newcommand\Z{\mathbb{Z}}
\newcommand\R{\mathbb{R}}
\newcommand\C{\mathbb{C}}
\newcommand\Q{\mathbb{Q}}

\newcommand\PK{\pi_K}

\newcommand\PJ{\pi_J}

\newcommand\PC{\pi_C}

\newcommand\EK{{E_K}}
\newcommand\EJ{{E_J}}

\newcommand\EC{{E_C}}
\newcommand\mK{\mu_K}
\newcommand\mJ{\mu_J}

\newcommand\laK{\lambda_K}
\newcommand\laJ{\lambda_J}

\newcommand\g{\gamma}
\newcommand\gK{\gamma_K}
\newcommand\gJ{\gamma_J}

\newcommand\G{\Gamma}

\newcommand\DKn{{\Gamma^{(n)} \gamma_K}}
\newcommand\DJn{{\Gamma^{(n)} \gamma_J}}

\newcommand\DPn{{\Gamma^{(n)} \gamma_P}}

\newcommand\DEn{{\Gamma^{(n)} \gamma_E}}

\newcommand\DKnn{{\Gamma^{(n+1)} \gamma_K}}
\newcommand\DJnn{{\Gamma^{(n+1)} \gamma_J}}

\newcommand\DPnn{{\Gamma^{(n+1)} \gamma_P}}

\newcommand\DEnn{{\Gamma^{(n+1)} \gamma_E}}

\newcommand\DPi{{\Gamma^{(i)} \gamma_P}}
\newcommand\DQi{{\Gamma^{(i)} \gamma_Q}}
\newcommand\DEi{{\Gamma^{(i)} \gamma_E}}

\title{Twisted homology cobordism invariants of knots in aspherical manifolds}
\author{Prudence Heck}
\date{\today}                                           

\begin{document}
\maketitle 

\begin{abstract}

We fix a null-homologous, homotopically essential knot $J$ in a 3-manifold with PTFA fundamental group and study concordance of knots that are homotopic to $J$.  We construct an infinite family of knots that are characteristic to $J$, and then use $L^2$-methods to show that they are not concordant to $J$.

\end{abstract}

\section{Introduction}

Let $M$ be a closed, irreducible, oriented, aspherical 3-manifold with poly-torsion-free-abelian (PTFA) fundamental group and let $J$ be a null-homologous, homotopically essential knot (that is, $J$ together with a tail to the base point of $M$ represents a nontrivial class of $\pi_1(M)$).  
A knot $K$ in $M$ is concordant to $J$ if together $J$ and $K$ bound an annulus in $M\times I$ that intersects the boundary transversely.  The aim of this paper is to fix a knot $J$ and construct obstructions to a knot $K$ being concordant to $J$ using $L^2$-methods first defined by T. Cochran, K. Orr, and P. Teichner in \cite{cot}, and then used by J. C. Cha, T. Cochran, S. Friedl, S. Harvey, T. Kim, C. Leidy, K. Orr, and P. Teichner, to study concordance of knots and links in $S^3$, and then to construct infinitely-many knots distinguished by these invariants.  For more on manifolds with solvable fundamental group see \cite{evans-moser}.

Our concordance obstructions are von-Neumann $\rho$-invariants of certain 3-manifolds $M(K)$ with coefficients in groups reminiscent of the solvable quotients of $\pi_1\left( M(K) \right)$ by its derived series.  The von-Neumann $\rho$-invariant, defined by J. Cheeger and M. Gromov in \cite{cheegerg-gromov}, is an oriented homeomorphism invariant that associates a real number to any regular cover of a closed, oriented 3-manifold $N$.  An important result of M. Ramachandran \cite{ram} is that if $N$ bounds an oriented 4-manifold $W$, and if the coefficient system $\g: \pi_1(N) \to A$ on $N$ extends over $\pi_1(W)$, then this $\rho$-invariant is equal to $\sigma^{(2)}(W, \g) - \sigma(W)$, where $\sigma^{(2)}(W, \g)$ is the $L^2$-signature of the intersection form on $H_2\left( W; \mathcal{U}(A)\right)$ and $\sigma(W)$ is the ordinary signature.  Given a knot $K$, the 3-manifold we use, $M(K)$, is a variation on zero-surgery on $\EK$, the exterior of $K$ in $M$.  The coefficient systems we use arise from a choice of localization similar to the algebraic closure of groups defined by J. Levine in \cite{alg.closureII}.  We note that Levine defined the algebraic closure in an attempt to distinguish {\em links} in $S^3$.  
In fact, the algebraic closure of the fundamental group of the exterior of a {\em knot} in $S^3$ is $\Z$, and therefore uninteresting.  The idea that classical invariants for links in $S^3$ could be used to distinguish knots in non-simply connected manifolds is due to D. Miller.  In \cite{miller2} he extended Milnor's $\overline{\mu}$-invariants for links in $S^3$ to knots in Seifert fibered manifolds that are homotopic to the Seifert fiber, and used them to study such knots up to concordance.

We associate to $K$ the closed 3-manifold $M(K) = \EK \cup \EJ$, called {\em $J$-surgery on $K$}, where $\partial \EK$ is identified with $\partial \EJ$ by $\mK \sim \mJ^{-1}$ and $\laK \sim \laJ$.  We use this 3-manifold instead of 0-surgery on $\EK$ for two reasons.  First, if $K$ is concordant to $J$ via an annulus $C \subset M \times I$ and if $\EC$ is the exterior of $C$ then $\partial \EC \cong M(K)$.  Second, if $G = \pi_1(M)$ then the inclusion of $\EK$ into $M$ induces an epimorphism $\gK: \pi_1(\EK) \to G$, and our construction requires that there be an epimorphism $\g_P: \pi_1 \left( M(K) \right) \to G$ as well.  If we used 0-surgery in place of $M(K)$ then no such epimorphism exists.  However, $\pi_1 \left( M(K) \right)$ is a pushout as in the following diagram, 
$$
\xymatrix{
\pi_1(T^2) \ar[r] \ar[d] & \pi_1(\EJ) \ar[d] \ar@/^/[ddr]^{\gJ} \\
\pi_1(\EK) \ar[r] \ar@/_/[drr]_{\gK} & \pi_1 \left( M(K) \right) \ar@{..>}[dr]|{\g_P} \\
& & G
}
$$
so the epimorphism $\g_P: \pi_1 \left( M(K) \right) \to G$ is uniquely defined (in fact, this epimorphism factors through $G \ast_{\Z} G$, so $M(K)$ can be regarded as a space over $G$ or $G \ast_{\Z} G$). 

Let $\mathcal{G}^G$ be the category whose objects are group epimorphisms onto $G$, $\g_A: A \to G$, and whose morphisms $f: \g_A \to \g_B$ are group homomorphisms over $G$.  Following Levine's construction of algebraic closure \cite{alg.closureII}, we construct a localization on $\mathcal{G}^G$ and use this to define for each object $\g_A$ of $\mathcal{G}^G$ the {\em rational $G$-local derived series}, a normal series
$$\xymatrix{ A \ar@{}[r]|{\unrhd} & \G \g_A  \ar@{}[r]|{\unrhd} & \G^{(1)} \g_A  \ar@{}[r]|{\unrhd} & \cdots  \ar@{}[r]|{\unrhd} & \G^{(n)} \g_A  \ar@{}[r]|{\unrhd} & \cdots }$$
reminiscent of the derived series.  This series is defined so that $\G \g_A = \text{Ker}(\g_A)$ and, in particular, so that $\g_A$ induces an epimorphism $\dfrac{A}{\G^{(n)} \g_A} \to G$ for all integers $n$. 

\begin{thrm}
Denote by $\Omega^G$ the class of morphisms in $\mathcal{G}^G$ satisfying the following properties:
\begin{enumerate}
\item $f: \g_A \rightarrow \g_B$ is a morphism with $A$ finitely generated and $B$ finitely presented,
\item $\G \g_{A}$ and $\G \g_{B}$ are finitely normally generated in $A$ and $B$, respectively,
\item $f$ induces a normal surjection $\G \g_{A} \rightarrow \G \g_{B}$, and
\item $f_{i}: H_{i}(A; \Z[G]) \rightarrow H_{i}(B; \Z[G])$ is an isomorphism for $i=1$ and an epimorphism for $i = 2$.
\end{enumerate}
There is a localization $(E, p)$ on $\mathcal{G}^G$ under which elements of $\Omega^G$ become isomorphisms.
\end{thrm} 

This leads to the following injectivity result, a necessary result for our construction,

\begin{lem}
If $f: \g_A \rightarrow \g_B$ is in $\Omega^G$ then $f$ induces a monomorphism $\overline{f}: \dfrac{A}{\G^{(n)} \g_A} \to \dfrac{B}{\G^{(n)} \g_B}$ for all $n$.
\end{lem}

The significance of the class $\Omega^G$ is that if $K$ is concordant to $J$ via an annulus $C$ then the inclusion $\EK \to \EC$ induces a homomorphism on fundamental groups that is a morphism in $\Omega^G$.  This observation together with the above lemma gives 

\begin{thrm}
Suppose that $K$ and $L$ are concordant knots and that $G$ is PTFA.  If $P = \pi_1 \left( M(K) \right)$ and $Q = \pi_1 \left( M(L) \right)$, and if $\g_A^n: A \to \dfrac{A}{\G^{(n)} \g_A}$ for $A \in \{P, Q\}$, 
then $\rho \left( M(K), \g_P^n \right) = \rho \left( M(L), \g_Q^n \right)$ for all $n$, regarded as spaces over $G$ or $G \ast_{\Z} G$.
\end{thrm}

A knot $K$ in $M$ is {\em $J$-characteristic} if there is a continuous map $\alpha: M \to M$ such that $\alpha(K) = J$ and $\alpha(M-K) \subseteq M-J$.  Miller's aforementioned invariants \cite{miller2}, which extend Milnor's $\overline{\mu}$-invariants, obstruct a knot being $J$-characteristic in the case that $J$ is the Seifert fiber of a (suitable) Seifert fibered 3-manifold.  Closely following the construction of Harvey in \cite{harvey2}, we construct infinitely many knots that are $J$-characteristic but not concordant to $J$.  Indeed, the $L^2$-signatures of the knots constructed here form a dense subset of $\R$.  We reiterate these remarks in the following theorem.

\begin{thrm}
Suppose that $\eta \in \DJn - \DJnn$ bounds an embedded disk in $M$.  There are knots $K_{\epsilon}$ such that if $P = \pi_1 \left( M(J) \right)$ and $Q_{\epsilon} = \pi_1 \left( M(K_{\epsilon}) \right)$ then
\begin{enumerate}
\item[(i)] Each $K_{\epsilon}$ is $J$-characteristic,
\item[(ii)] For each $K_{\epsilon}$, $\rho(M(K_{\epsilon}), \g_{Q_{\epsilon}}^i) = \rho(M(J), \g_P^i)$ for $i \leq n$, and 
\item[(iii)] $\{ \rho(M(K_{\epsilon}), \g_{Q_{\epsilon}}^{n+1}) \}$ is a dense subset of $\R$.
\end{enumerate}
\end{thrm}

The paper is organized as follows.  In Section \ref{Preliminaries} we review some basic definitions from knot theory.  We also show that the exteriors of the knots under consideration are Eilenberg-MacLane spaces.  In Section \ref{Characteristic series and an injectivity theorem} we introduce the category $\mathcal{G}^G$, construct the {\em rational $G$-local derived series}, and prove the above injectivity lemma.  Section \ref{Review of L^2 methods} is a review of $L^2$-signatures necessary for our main results.  In the first half of Section \ref{J-surgery and analytic invariants} we define {\em $J$-surgery on a knot $K$} and show that certain $\rho$-invariants of $J$-surgery are concordance invariants.  In the second half of Section \ref{J-surgery and analytic invariants} we construct infinitely many non-concordant $J$-characteristic knots and distinguish them up to concordance via $\rho$-invariants of their $J$-surgeries.  Finally, Section \ref{const of localization} is devoted to constructing the algebraic closure necessary for the results of Section \ref{Characteristic series and an injectivity theorem}. \\

The results presented here constitute part of the author's Ph.D. thesis.  She is especially thankful to her advisor, Kent Orr, for all of his support and guidance, and for many useful discussions on this work.

\section{Preliminaries}
\label{Preliminaries}

Henceforth we work in the smooth category and only consider null-homologous knots in closed, oriented, irreducible, aspherical manifolds that, when based, represent a homotopy class of infinite order.  However, many of the definitions in this chapter extend to null-homologous knots in orientable 3-manifolds.  Recommended references are \cite{gompf-stip}, \cite{munkres}, and \cite{rolfsen}. \\

A {\em knot} $K$ in a 3-manifold $M$ is an oriented one-dimensional closed submanifold.  We will assume that $M$ has base point $p \in M$ and that $K$ does not contain $p$.  We also assume that $K$ is based via an embedded path to $p$, although the choice of path will not matter for the results of this paper.  The {\em exterior} of $K$ is the complement of an open normal neighborhood of $K$, $\EK = M-N(K)$, based at $p$.  Note that the inclusion $i_K: \EK \hookrightarrow M$ induces an epimorphism $\gK: \pi_1 (\EK, p) \rightarrow \pi_1 (M, p)$.  We denote the fundamental group of $\EK$ by $\PK$ and the fundamental group of $M$ by $G$. \\

A {\em meridian} $\mK$ of $K$ is an embedded curve in $\partial \EK$ that represents a primitive element of $H_1(\partial \EK)$ and bounds a disk in $N(K)$.  It is uniquely determined up to isotopy in $\partial \EK$.  We will abuse notation by letting $\mK$ denote the curve in $M$, its isotopy class in $\partial \EK$, its homology class in $H_1(\EK)$, and its homotopy class in $\PK$, where in the last case we regard $\mK$ as being based via the path basing $K$. \\

If $K$ is a null-homologous knot in $M$ then it can be shown that
$$\xymatrix{ \text{Ker}\{ H_1(\EK) \ar[r] & H_1(M) \} \cong \Z, }$$
generated by $\mK$.  If $L$ is another null-homologous knot in $M$ that is disjoint from $K$ then
$$\xymatrix{ [L] = n\mK \in \text{Ker}\{ H_1(\EK) \ar[r] & H_1(M) \} }$$
for some unique $n$.  Define the {\em linking number} of $K$ and $L$ by $lk(K, L) = n$.  An embedded curve $\laK$ in $\partial \EK$ is a {\em longitude} of $K$ if it represents a nontrivial element of $H_1(\partial \EK)$ and $lk(K, \laK) = 0$.  A longitude always exists (for null-homologous knots) and is unique up to isotopy in $\partial \EK$. \\

A choice of meridian $\mK$ and longitude $\laK$ determines, up to isotopy, a framing of $K$
$$t_K: \xymatrix{ S^1 \times D^2 \ar[r] & N(K) }$$
by  $t_K( \{1\} \times \partial D^2) = \mK$ and $t_K(S^1 \times \{1\}) = \laK$.  Since $M$ is oriented, we regard $S^1 \times D^2$ as having a fixed orientation and require that $t_K$ be orientation-preserving.  We call (the isotopy class of) $t_K$ the {\em canonical framing} determined by $\mK$ and $\laK$. \\

A knot $K$ in $M$ is {\em $L$-characteristic} if there is a continuous map $\alpha: M \to M$ such that $\alpha(K) = L$ and $\alpha(M-K) \subseteq M-L$. \\

Two knots $K$ and $L$ are said to be {\em concordant} if there is an oriented submanifold $C$ of $M\times I$ that is homeomorphic to $S^1 \times I$, meets the boundary of $M\times I$ transversely with $C \cap (M\times \{0\}) = L \times\{0\}$ and $C \cap (M\times \{1\}) = K \times\{1\}$, and such that the orientations of $L \times \{0\}$ and $K \times \{1\}$ agree with the orientations induced by $C$. \\

If $K$ and $L$ are concordant then we can always choose $C$ to be disjoint from $\{p\} \times I$.  We call $E_C = M\times I - N(C)$ the {\em exterior} of $C$ in $M\times I$, where $N(C)$ is an open normal neighborhood of $C$ that is disjoint from $\{p\} \times I$.  
The path $\rho(t) := p \times t$ induces an isomorphism $\rho_*: \pi_1(\EC, p\times 1) \to \pi_1(\EC, p\times 0)$ by $\rho_*(f) = \rho * f * \rho^{-1}$ (where we read concatenation of paths from left to right).  We therefore write $\pi_1(\EC, p\times 1) = \pi_1(\EC, p\times 0)$ and denote both of these groups by $\PC$. \\

If $C$ is a concordance between $K$ and $L$ then the inclusions $\iota_L : E_L \hookrightarrow E_C$ and $\iota_K : E_K \hookrightarrow E_C$ induce homomorphisms $\iota_L : \pi_L \to \pi_C$ and $\iota_K : \pi_K \to \pi_C$, respectively, and the inclusion $i_C: E_C \hookrightarrow M\times I$ induces an epimorphism $\g_C : \PC \rightarrow G = \pi_1(M)$ such that the following diagram commutes
$$\xymatrix{
\pi_{L} \ar[r]^{\iota_{L}} \ar[rd]_{\g_{L}} & \PC \ar[d]^{\g_C} & \pi_{K} \ar[l]_{\iota_{K}} \ar[ld]^{\g_{K}}\\
& \M.
}
$$


\begin{prop}
\label{conc is hmlgy conc}
Let $C$ be a concordance between two knots $K$ and $L$.  Then the inclusions $\iota_L : E_L \hookrightarrow E_C$ and $\iota_K : E_K \hookrightarrow E_C$ induce isomorphisms on $H_*( -; \Z[\M])$.
\end{prop}
\begin{proof}
It is enough to prove the result for $K$.  Choose a path from $p\times 1$ to $K$; this will induce coefficient systems on $\partial \EK$ and $\partial \EC$.  The inclusion $\iota_K: \EK \to \EC$ induces the following diagram of Mayer-Vietoris sequences:
$$
\xymatrix{
H_{n+1}\left( M; \Z[\M] \right) \ar[r] \ar[d] & H_n\left( \partial \EK; \Z[\M] \right) \ar[r] \ar[d] & \text{$\phantom{MMMMM}$} \\
H_{n+1}\left( M\times I ; \Z[\M] \right) \ar[r] & H_n\left( \partial \EC; \Z[\M] \right) \ar[r] & \text{$\phantom{MMMMM}$} \\
& H_n\left( \EK; \Z[\M] \right) \oplus H_n\left( N(K) \Z[\M] \right) \ar[r] \ar[d] & H_n\left( M; \Z[\M] \right) \ar[d] \\
& H_n\left( \EC; \Z[\M] \right) \oplus H_n\left( N(C); \Z[\M] \right) \ar[r] & H_n\left( M\times I ; \Z[\M] \right)
}
$$
Since $M$ is aspherical and $\M$ is the fundamental group of $M$,
$$H_{n+1}\left( M; \Z[\M] \right) = H_{n}\left( M; \Z[\M] \right) = H_{n+1}\left( M\times I; \Z[\M] \right) = H_{n}\left( M \times I; \Z[\M] \right) = 0$$
for $n>0$.  The inclusions $\partial \EK \hookrightarrow \partial \EC$ and $N(K) \hookrightarrow N(C)$ are homotopy equivalences, as are their lifts to the $\M$-cover, and therefore induce isomorphisms on $H_n\left(-; \Z[\M] \right)$.  It follows from a diagram chase that $(\iota_K)_*: H_n\left( \EK; \Z[\M] \right) \to H_n\left( \EC; \Z[\M] \right)$ is an isomorphism for all $n > 0$.  For $n = 0$, the $G$-covers of $\EK$ and $\EC$ are connected since the coefficient systems $\gK: \PK \to G$ and $\g_C: \pi_C \to G$, respectively, are epimorphisms.  Hence, $(\iota_K)_*: H_0\left( \EK; \Z[\M] \right) \to H_0\left( \EC; \Z[\M] \right)$ is an isomorphism as well.
\end{proof}

\begin{prop}
\label{basics of \EK}
If $K$ is a null-homologous knot in $M$ that, when based, represents a homotopy class of infinite order then
\begin{enumerate}
\item $\EK$ is an Eilenberg-MacLane space, 

\item The inclusion $\partial \EK \hookrightarrow \EK$ induces isomorphisms 
$$\xymatrix{ H_i (\partial \EK; \Z[\M]) \cong H_i (\EK; \Z[\M]) }$$ 
for all $i\geq 1$,  

\item $H_i (\EK; \Z[\M]) = 0$ for all $i\geq 2$ and 
$$\xymatrix{ H_1 \left( \partial \EK; \Z[\M] \right) \cong H_1\left( \mK; \Z\left[ \frac{\M}{[K]\Z} \right] \right) = \bigoplus_{\frac{\M}{[K]\Z}} \Z }$$ 
as a $\Z[\M]$-module, where $[K]\Z \subset \M$ is the infinite cyclic subgroup generated by $[K]$ and $\dfrac{\M}{[K]\Z}$ is the set of right cosets of $[K]\Z$ in $\M$. 
\end{enumerate}
\end{prop}

\begin{proof}
(1)  The universal cover of $\EK$ is a simply connected 3-manifold with boundary and therefore has trivial homology in all dimensions greater than two.  We will show that $\pi_2 (\EK) = 0$, implying that the second homotopy group of the universal cover of $\EK$ is trivial.  Statement (1) then follows from Hopf's Theorem.  

Suppose that $\pi_2(\EK) \neq 0$.  By the Sphere Theorem there is an embedded sphere $S$ in $\EK$ representing a nonzero homotopy class of $\EK$.  Then $S$ is a sphere in $M$ and, since $M$ is irreducible, $S$ bounds an embedded ball $B \subset M$.  If $K$ were contained in $B$ then $[K]$ would not represent an element of infinite order in $\pi_1(M)$.  It follows that $B \subset \EK$, and therefore that $\pi_2 (\EK) = 0$.

(2)  From the $\Z[\M]$-coefficient Mayer-Vietoris sequence 
$$\xymatrix{ H_{n+1}\left( M; \Z[\M] \right) \ar[r] & H_{n}\left( \partial \EK; \Z[\M] \right) \ar[r] & \txt{$H_{n}\left( \EK; \Z[\M] \right)$\\ $\oplus$\\ $H_n \left( N(K); \Z[\M] \right)$} \ar[r] & H_{n}\left( M; \Z[\M] \right) }$$ 
it follows that 
$$ H_i \left( \partial \EK; \Z[\M] \right) \cong H_{i}\left( \EK; \Z[\M] \right) \oplus H_i \left(N(K); \Z[\M] \right) $$ 
for all $i\geq 1$.  As $[K]$ is of infinite order in $\M$, $H_i \left( N(K) ; \Z[\M] \right) = 0$ for $i\geq 1$.  Hence, the inclusion $\partial \EK \hookrightarrow \EK$ induces an isomorphism $H_i \left( \partial \EK; \Z[\M] \right) \cong H_i \left( \EK; \Z[\M] \right)$ for all $i\geq 1$.

(3)  Since $[K]$ is of infinite order in $\M$,
$$ H_i \left( \partial \EK; \Z \left[ \M \right] \right) \cong H_i \left( \mK; \Z \left[ \frac{\M}{[K]\Z} \right] \right) $$ 
for all $i \geq 0$, where $[K]\Z$ is the infinite cyclic subgroup of $\M$ generated by $[K]$ and $\dfrac{\M}{[K]\Z}$ is the set of right cosets of $[K]\Z$.  The coefficient system on $\mK$ is trivial because the image of $\mK$ in $\M$ is trivial.  Hence, as a $\Z[\M]$-module this equals zero if $i>1$ and $\bigoplus_{\frac{\M}{[K]\Z} } \Z$ if $i=1$.
\end{proof}

\section{Characteristic series and an injectivity lemma}
\label{Characteristic series and an injectivity theorem}

\subsection{Rational $G$-derived series}

\begin{defn}{}
For a fixed group $\M$, define the {\em category of groups over $\M$}, denoted $\mathcal{G}^{\M}$, to be the category whose objects are surjective homomorphisms of groups $\g_A:A \rightarrow G$ and whose morphisms are homomorphisms of groups $f: A \rightarrow B$ making the following diagram commute.
$$\xymatrix{
A \ar[dr]_{\g_A}\ar[rr]^f & & B \ar[dl]^{\g_B}\\
& \M
}
$$
Notice that if $\M = \{e\}$ we recover the category of groups.
\end{defn}

\begin{defn}
\label{Gamma}
For any object $\g_A$ of $\mathcal{G}^G$, we define 
$$\G \g_A = \text{Ker}\left\{ \g_A \right\}.$$
\end{defn}

\begin{defn}
\label{def of rat derived}
Define the {\em rational G-derived series} of an object $\g_A$ of $\mathcal{G}^G$ by
$$\G^{(0)}_r \g_A := \G \g_A \text{$\phantom{MM}$ and $\phantom{MM}$} \G^{(n+1)}_r \g_A := \text{Ker} \left\{ \G^{(n)}_r \g_A \to (\G^{(n)}_r \g_A)_{ab}\otimes_{\Z} \Q \right\}, $$ \\
where $(\G^{(n)}_r \g_A)_{ab}$ is the abelianization of $\G^{(n)}_r \g_A$.
\end{defn}

\begin{lem}
\label{derived is normal}
Each $\G^{(n)}_r \g_A$ is a normal subgroup of $A$, and the commutator subgroup $\left(\G \g_A\right)^{(n)}$is contained in $\G^{(n)}_r \g_A$.
\end{lem}
\begin{proof}$\G^{(0)}_r \g_A$ is clearly normal in $A$.  Suppose by induction that $\G^{(n)}_r \g_A$ is normal in $A$ and choose $a\in \G^{(n+1)}_r \g_A$ and $g\in A$.  As $\G^{(n)}_r \g_A$ is normal, $g^{-1}ag \in \G^{(n)}_r \g_A$.  Also, $a \in \G^{(n+1)}_r \g_A$ implies that there is an integer $m$ such that 
$$a^m \in [\G^{(n)}_r \g_A, \G^{(n)}_r \g_A].$$  
Since $(g^{-1}ag)^m = g^{-1}a^mg$, it follows that $g^{-1}ag$ is in the kernel of
$$\xymatrix{ \G^{(n)}_r \g_A \ar[r] &  \left( \G^{(n)}_r \g_A \right)_{ab} \otimes_{\Z} \Q \phantom{a}. }$$
Hence, $\G^{(n+1)}_r \g_A$ is normal in $A$. 

For the second statement, $\left(\G \g_A\right)^{(0)} = \G \g_A = \G^{(0)}_r \g_A$.  If $\left(\G \g_A\right)^{(n)} \subset \G^{(n)}_r \g_A$ then
$$\xymatrix@1{ \left(\G \g_A\right)^{(n+1)} = \left[ \left(\G \g_A\right)^{(n)}, \left(\G \g_A\right)^{(n)} \right] \subset \left[ \G^{(n)}_r \g_A, \G^{(n)}_r \g_A \right] \subset \G^{(n+1)}_r \g_A. }$$
\end{proof}

\begin{example}
Consider the following short exact sequence
$$\xymatrix{0\ar[r] & F_2 \ar[r]^(.4){\iota} & F\times \Z \ar[r]^(.5){\g} & \dfrac{F}{F_3} \ar[r] & 0, }$$
where $F$ is the free group on two generators $x$ and $y$, $\Z$ is generated by $t$, and $F_2$ is the commutator subgroup of $F$.  Define $\iota: F_2 \rightarrow F\times \Z$ by $\iota([x, y]^{\omega}) = [x, y]^{\omega}t^{-1}$ for any word $\omega \in F\times \Z$.  Then $\G \g \cong F_2$ is the free group on infinitely many generators and, because the quotients of the commutator series $\dfrac{(F_2)^{(n)}}{(F_2)^{(n+1)} }$ are $\Z$-torsion free for all $n$, $\G^{(n)}_r \g \cong (F_2)^{(n)}$.  One can show that $\dfrac{F}{F_3}$ is the fundamental group of the Heisenberg manifold, a circle bundle over the torus, and that $F \times \Z$ is the fundamental group of the complement of a fiber in this manifold. 
\end{example}

\begin{lem}
\label{rat is funct}
If $f: \g_A \to \g_B$ is a morphism in $\mathcal{G}^G$ then $f\left( \G^{(n)}_r \g_A \right) \subset \G^{(n)}_r \g_B$ for all $n$.
\end{lem}
\begin{proof}
The proof is by induction on $n$.  First note that $f\left( \G^{(0)}_r \g_A \right) \subset \G^{(0)}_r \g_B$ because $f$ is a morphism in $\mathcal{G}^G$.  Suppose that $f\left( \G^{(n)}_r \g_A \right) \subset \G^{(n)}_r \g_B$ and let $a \in \G^{(n+1)}_r \g_A$.  By the definition of $\G^{(n+1)}_r \g_A$ there is some $m \in \Z$ such that $a^m \in [\G^{(n)}_r \g_A, \G^{(n)}_r \g_A]$.  Hence, $f(a)^m = f(a^m)$ is in $[\G^{(n)}_r \g_B, \G^{(n)}_r \g_B]$, and it follows that $f(a) \in \G^{(n+1)}_r \g_B$.
\end{proof}

\begin{defn}
A group $G$ is {\em poly-torsion-free abelian} (PTFA) if it admits a normal series
$$\xymatrix@1{G = G_0 \rhd G_1 \rhd \cdots \rhd G_n = \{e\} }$$
for some $n\in \N$ with $\dfrac{G_i}{G_{i+1}}$ torsion-free abelian for all $i < n$. 
\end{defn}

\begin{remark} 
\label{PTFA example} $\phantom{}$
\begin{enumerate}
\item Any torsion-free abelian group is PTFA. \\
\item If $G$ is PTFA then it has no element of finite order.  \\
\item If $0 \rightarrow N \rightarrow P \rightarrow Q \rightarrow 0$ is a short exact sequence with $N$ torsion-free abelian and $Q$ PTFA then $P$ is PTFA. \\
\item Any subgroup of a PTFA group is PTFA.
\end{enumerate}
\end{remark}

\begin{lem}
\label{rat derived is PTFA}
If $G$ is PTFA then so is $\dfrac{A}{\G^{(n)}_r \g_{A}}$ for all $n$.
\end{lem}
\begin{proof}
If $G$ is PTFA then so is $\dfrac{A}{\G^{(0)}_r \g_{A}} \cong G$.  Suppose by induction that $\dfrac{A}{\G^{(n)}_r \g_{A}}$ is PTFA and consider the following short exact sequence
$$
\xymatrix{
0 \ar[r] & \dfrac{\G^{(n)}_r \g_{A}}{\G^{(n+1)}_r \g_{A}} \ar[r] & \dfrac{A}{\G^{(n+1)}_r \g_{A}} \ar[r] & \dfrac{A}{\G^{(n)}_r \g_{A}} \ar[r] & 0.
}
$$
The kernel $\dfrac{\G^{(n)}_r \g_{A}}{\G^{(n+1)}_r \g_{A}}$ is torsion-free abelian by construction.  Hence, $\dfrac{A}{\G^{(n+1)}_r \g_{A}}$ is PTFA by (3) of Remark \ref{PTFA example}.
\end{proof}

\subsection{The rational $G$-local derived series}
We now construct the {\em rational $G$-local derived series} and obtain our injectivity lemma, Lemma \ref{quotient monomorphisms}.  The proof of Theorem \ref{M-localization}, which asserts the existence of a localization on $\mathcal{G}^G$, can be found in Section \ref{const of localization}.

\begin{defn}
\label{Omega}
Denote by $\Omega^G$ the class of morphisms in $\mathcal{G}^G$ satisfying the following properties:
\begin{enumerate}
\item $f: \g_A \rightarrow \g_B$ is a morphism with $A$ finitely generated and $B$ finitely presented,
\item $\G \g_{A}$ and $\G \g_{B}$ are finitely normally generated in $A$ and $B$, respectively,
\item $f$ induces a normal surjection $\G \g_{A} \rightarrow \G \g_{B}$, and
\item $f_{i}: H_{i}(A; \Z[G]) \rightarrow H_{i}(B; \Z[G])$ is an isomorphism for $i=1$ and an epimorphism for $i = 2$. 
\end{enumerate}
\end{defn}

\begin{defn}
\label{local object}
An object $\g_X \in \mathcal{G}^G$ is called {\em $\Omega^G$-local} if for any morphism $\g_A \rightarrow \g_B$ in $\Omega^G$ and any morphism $\g_A \rightarrow \g_X$, there is a unique morphism $\g_B \rightarrow \g_X$ making the following diagram commute:
$$
\xymatrix{
A \ar[r] \ar[d] & B \ar[dl]\\
X 
}$$
\end{defn}

\begin{defn}
A \emph{localization} is a pair $(E, p)$, where $E: \mathcal{G}^G \rightarrow \mathcal{G}^G$ is a functor and $p: id_{\mathcal{G}^G} \rightarrow E$ is a natural transformation such that for any morphism $\g_A \rightarrow \g_X$ into a local object $\g_X$ there is a unique morphism $E(\g_A) \rightarrow \g_X$ making the following diagram commute,
$$
\xymatrix{
A \ar[r] \ar[d] & E(A) \ar[dl] \phantom{a} \\
X
}
$$
where $E(\g_A): E(A) \rightarrow G$.  To be consistent with the notation of $\mathcal{G}^G$ we write $\g_{E(A)}$ for $E(\g_A)$ from now on. 
\end{defn}

\begin{thrm}
\label{M-localization} 
There is a localization $(E, p)$ on $\mathcal{G}^G$ under which elements of $\Omega^G$ become isomorphisms.
\end{thrm}
\begin{proof}
The proof can be found in Section \ref{const of localization}.
\end{proof}

\begin{defn}
\label{rat loc derived}
Recall from Definition \ref{Gamma} that $\G \g_A = \text{Ker}\left\{ \g_A \right\}$.  Define the {\em rational $G$-local derived series} of an object $\g_A$ in $\mathcal{G}^G$ by
$$\G^{(0)}\g_A := \G \g_A$$
and
$$\G^{(n)}\g_A := \text{Ker}\left\{ A \to \dfrac{E(A)}{\G^{(n)}_r \g_{E(A)}} \right\} .$$
For $\g_A: A \to G$ this defines a normal series
$$\xymatrix{ A \ar@{}[r]|{\unrhd} & \G \g_A  \ar@{}[r]|{\unrhd} & \G^{(1)} \g_A  \ar@{}[r]|{\unrhd} & \cdots  \ar@{}[r]|{\unrhd} & \G^{(n)} \g_A  \ar@{}[r]|{\unrhd} & \cdots . }$$
\end{defn}

\begin{lem}
\label{loc is PTFA}
If $G$ is PTFA then so is $\dfrac{A}{\G^{(n)}\g_A}$ for all $n$.
\end{lem}
\begin{proof}
Each $\G^{(n)}\g_A$ is normal as, by definition, it is the kernel of a homomorphism on $A$.  By construction, the homomorphism $A \to \dfrac{E(A)}{\G^{(n)}_r \g_{E(A)}}$ induces a monomorphism
$$\dfrac{A}{\G^{(n)}\g_A} \to \dfrac{E(A)}{\G_r^{(n)}\g_{E(A)}}. $$
Suppose that $G$ is PTFA.  Then $\dfrac{E(A)}{\G_r^{(n)}\g_{E(A)}}$ is PTFA by Lemma \ref{rat derived is PTFA} and, because the class of PTFA groups is closed under subgroups (as noted in (4) of Remark \ref{PTFA example}), it follows that $\dfrac{A}{\G^{(n)}\g_A}$ is PTFA . 
\end{proof}

\begin{lem}
\label{derived to derived}
If $f: \g_A \rightarrow \g_B$ is a morphism in $\mathcal{G}^G$ then, for all $n$, $f(\G^{(n)} \g_A) \subset \G^{(n)} \g_B$. 
\end{lem}
\begin{proof}
The morphism $E(f): \g_{E(A)} \to \g_{E(B)}$ induces a homomorphism 
$$\overline{E(f)}: \dfrac{E(A)}{ \G^{(n)}_r \g_{E(A)} } \to \dfrac{E(B)}{ \G^{(n)}_r \g_{E(B)} }$$
for all $n$ by Lemma \ref{rat is funct}.  We therefore have the following commutative diagram:
$$
\xymatrix{
A \ar[r]^f \ar[d] & B \ar[d] \\
E(A) \ar[r]^{E(f)} \ar[d] & E(B) \ar[d] \\
\dfrac{E(A)}{\G^{(n)}_r \g_{E(A)}} \ar[r]^{\overline{E(f)}} & \dfrac{E(B)}{\G^{(n)}_r \g_{E(B)}}
}
$$
By definition, $\G^{(n)} \g_A$ and $\G^{(n)} \g_B$ are the respective kernels of the (composite) vertical homomorphisms.  It follows from a simple diagram chase that $f(\G^{(n)} \g_A) \subset \G^{(n)} \g_B$. 
\end{proof}

\begin{lem}
\label{quotient monomorphisms}
If $f: \g_A \rightarrow \g_B$ is in $\Omega^G$ then $f$ induces a monomorphism $\overline{f}: \dfrac{A}{\G^{(n)} \g_A} \to \dfrac{B}{\G^{(n)} \g_B}$ for all $n$.
\end{lem}
\begin{proof}
By Lemma \ref{derived to derived} we have the following commutative diagram:
$$
\xymatrix{
\dfrac{A}{\G^{(n)} \g_A} \ar[r]^{\overline{f}} \ar[d] & \dfrac{B}{\G^{(n)} \g_B} \ar[d] \\
\dfrac{E(A)}{\G^{(n)}_r \g_{E(A)}} \ar[r]^{\overline{E(f)}} & \dfrac{E(B)}{\G^{(n)}_r \g_{E(B)}}
}
$$
The vertical maps are monomorphisms because $\G^{(n)} \g_A$ and $\G^{(n)} \g_B$ are defined to be the kernels of $A \rightarrow \dfrac{E(A)}{\G^{(n)} \g_{E(A)}}$ and $B \rightarrow \dfrac{E(B)}{\G^{(n)} \g_{E(B)}}$, respectively.  Also, $E(f): E(A) \rightarrow E(B)$ is an isomorphism since $f$ is in $\Omega^G$.  The bottom map is therefore an isomorphism, and it follows that $\overline{f}$ is a monomorphism.
\end{proof}

\section{Review of $L^2$-methods} 
\label{Review of L^2 methods}

In this section we review $L^2$-signatures and the Cheeger-Gromov $\rho$-invariant.  Most of the section closely follows \cite{cot}, \cite{luck}, and \cite{reich}.  Other recommended references are \cite{FollandRealAnalysis} and \cite{ReedSimon}.  {\em In this section, and only this section, we use $G$ to denote an arbitrary group}.  In the next section we will be interested in $L^2$-signatures of operators in $\text{Herm}_n \left( \mathcal{U}\left( \frac{A}{\G^{(n)} \g_A} \right) \right)$ for an object $\g_A \in \mathcal{G}^G$. \\

Recall that a complex vector space with inner product is a Hilbert space if it is complete with respect to the norm induced by the inner product.  For a countable discrete group $G$, $l^2(G)$ is the Hilbert space of all square-summable formal sums over $G$ with complex coefficients.  An element of $l^2(G)$ is of the form
$$ \sum_{g\in G} z_g g$$
with $z_g \in \C$ and $\sum_{g\in G} |z_g|^2 < \infty$.  The inner product on $l^2(G)$ is given by
$$ \left\langle \sum_{g\in G} z_g g, \sum_{g\in G} v_g g \right\rangle = \sum_{g\in G} z_g \overline{v}_g. $$
The group $G$ acts isometrically on $l^2(G)$ on the left by multiplication, 
$$h \cdot  \sum_{g\in G} z_g g =  \sum_{g\in G} z_g hg. $$
Right multiplication induces an embedding of $\C G$ into $\mathcal{B}(l^2(G))$, the space of bounded operators on $l^2(G)$.  The {\em group von Neumann algebra} $\mathcal{N}(G)$ of $G$ is the algebra of left $G$-equivariant bounded operators on $l^2(G)$.  In particular, $\C G$ is a subalgebra of $\mathcal{N}(G)$.  The {\em von Neumann trace} $tr_G : \mathcal{N}(G) \to \C$ is defined by
$$tr_G(f) = \langle f(e), e \rangle_{l^2(G)} $$
where $e$ is the identity element in $G$.  The trace is symmetric in that $tr_G (fh) = tr_G (hf)$, and if $tr_G (f^* f) = 0$ then $f = 0$.  Extend $tr_G$ to $n\times n$-matrices over $\mathcal{N}(G)$ by
$$tr_G \left( (f_{ij})_{i, j=1}^n \right) = \sum_{i=1}^n tr_G(f_{ii}) $$ 
for $(f_{ij})_{i, j=1}^n \in M_n \left( \mathcal{N}(G) \right)$. 

\begin{defn}
Let $\mathcal{U}(G)$ denote the algebra of operators affiliated to $\mathcal{N}(G)$.  An element $f \in \mathcal{U}(G)$ is a possibly unbounded operator on $l^2(G)$ satisfying the following conditions: 
\begin{enumerate}
\item The domain of $f$, $dom(f)$, is dense in $l^2(G)$, 
\item The graph of $f$ is closed in $l^2(G) \times l^2(G)$, and 
\item For every $g \in \mathcal{N}(G)' = \left\{ g\in \mathcal{B}(l^2(G)) | gm = mg \text{ for all } m \in \mathcal{N}(G) \right\}$,
	$$dom(gf) \subset dom(fg)$$
	and $gf = fg$ on $dom(gf)$. 
\end{enumerate}
\end{defn}

The set $\mathcal{U}(G)$ is a $*$-algebra that contains $\mathcal{N}(G)$ as a subalgebra and a von Neumann regular ring (i.e. it contains ``weak inverses").  Moreover, it is canonically isomorphic to the classical right ring of fractions of $\mathcal{N}(G)$ with respect to all non-zero divisors.  Let $\text{Herm}_n\left(\mathcal{U}(G) \right)$ denote the set of self-adjoint $n\times n$-matrices with coefficients in $\mathcal{U}(G)$.  Any matrix $M \in \text{Herm}_n\left(\mathcal{U}(G) \right)$ is a self-adjoint (unbounded) operator on $l^2(G)^n$.
The characteristic functions $p_+$ and $p_-$ onto the positive and negative spectrum, respectively, are bounded Borel functions.  Hence, the functional calculus for unbounded self-adjoint operators and bounded Borel functions defines bounded operators $p_+(M)$ and $p_-(M)$ on $l^2(G)^n$.  That is, $p_+(M)$ and $p_-(M)$ are elements of $M_n\left(\mathcal{N}(G) \right)$. 

\begin{defn}
Define the $L^2$-signature of $M \in \text{Herm}_n\left(\mathcal{U}(G) \right)$ by
$$ \sigma^{(2)}_G (M) := tr_G(p_+(M)) - tr_G(p_-(M)). $$
The real number $\sigma^{(2)}_G (M)$ only depends on the $G$-isometry class of $M$ \cite{cot}. 
\end{defn}

The next definition makes use of the fact that $\mathcal{U}(G)$ is a von Neumann regular ring because then every finitely presented $\mathcal{U}(G)$-module is projective \cite{reich}.
\begin{defn}
Let $W$ be an oriented $4$-manifold with coefficient system $\g_W: \pi_1(W) \rightarrow G$.  The intersection form 
$$\xymatrix{ h_{\g_W}: H_2\left(W; \mathcal{U}(G) \right) \ar[r] & H_2\left(W, \partial W; \mathcal{U}(G) \right) \ar[r]^{PD} &  \text{$\phantom{MMMMM}$} \\
& H^2\left(W; \mathcal{U}(G)\right) \ar[r]^(.35){\kappa} & \text{Hom}_{\mathcal{U}(G)}\left( H_2\left(W; \mathcal{U}(G) \right), \mathcal{U}(G) \right) }$$
is an element of $\text{Herm}_n\left(\mathcal{U}(G) \right)$ and the {\em $L^2$-signature of $W$} is defined by
$$ \sigma^{(2)}(W, \g_W) := \sigma^{(2)}_G \left( h_{\g_M} \right) . $$ 
\end{defn}

Cheeger and Gromov defined the {\em von Neumann $\rho$-invariant}, an oriented homeomorphism invariant of regular covers of 3-manifolds \cite{cheegerg-gromov}.  If $M$ is a 3-manifold and $\g_M: \pi_1(M) \to G$ is a coefficient system then $\rho(M, \g_M)$ is a real number.  The following theorem is due to Ramachandran \cite{ram} (see also \cite{luck-schick}).  {\em Recall that in this section we use $G$ to denote an arbitrary group}.

\begin{thrm}
\label{ram}
If $W$ is a compact, oriented 4-manifold such that $M = \partial W$ and if $\g_M: \pi_1(M) \to G$ extends to $\g_W: \pi_1(W) \to G$ then 
$$ \rho(M, \g_M) = \sigma^{(2)}(W, \g_W) - \sigma(W), $$
where $\sigma(W)$ is the ordinary signature of $W$. 
\end{thrm}

This next theorem is a collection of results from and consequences of Chapter 5 of \cite{cot}.  We note that $(2)$ is essentially Atiyah's $L^2$-Index Theorem.
\begin{thrm}
\label{L^2 stuff}
Let $\sigma(W)$ denote the ordinary signature of a compact orientable 4-manifold $W$.  The $L^2$-signature has the following properties:
\begin{enumerate}
\item If $W$ is the boundary of a 5-dimensional manifold (with homomorphism extending $\g_W$) or if $G$ is PTFA and
$$H_2\left(W; \Q[G] \right) \to H_2\left(W, \partial W; \Q[G] \right) $$
is zero then $\sigma^{(2)}(W, \g_W) = 0$.

\item If $W$ is closed or if $\g_W$ is trivial then $\sigma^{(2)}(W, \g_W) = \sigma(W)$.

\item If $(W, \g_W)$ and $(\tilde{W}, \g_{\tilde{W}})$ have the same boundary $M$ and if $\g_W$ agrees with $\g_{\tilde{W}}$ on $M$ then
$$\sigma^{(2)}(W \cup_M \tilde{W}, \g_W \cup \g_{\tilde{W}}) = \sigma^{(2)}(W, \g_W) + \sigma^{(2)}(\tilde{W}, \g_{\tilde{W}}).$$

\item If $i: G \to H$ is a monomorphism then $\sigma^{(2)}(W, \g_W) = \sigma^{(2)}(W, i \circ \g_W)$ and $\rho(M, \g_M) = \rho(M, i \circ \g_M)$.

\item If $(M, \g_M)$ and $(\tilde{M}, \g_{\tilde{M}})$ are closed 3-manifolds then $\rho(M \coprod \tilde{M}, \g_M \coprod \g_{\tilde{M}}) = \rho(M, \g_M) + \rho(\tilde{M}, \g_{\tilde{M}})$. 
\end{enumerate}
\end{thrm}

\section{$J$-surgery and analytic invariants}
\label{J-surgery and analytic invariants}

\subsection{$J$-surgery}{}

Let $K$ be a knot in $M$ and suppose that $K$ is concordant to $J$ via a concordance $C$.  We assume that $K$ and $J$ are based and have the canonical framing.  Then
$$\partial \EC = \EK \cup T^2 \times I \cup -\EJ$$
where $T^2 \times \{0\}$ is identified with $\partial \EJ$, $T^2 \times \{1\}$ is identified with $\partial \EK$, and $-\EJ$ denotes $\EJ$ with the reverse orientation. 

\begin{defn}
\label{J surgery}
For any knot $K$ define {\em $J$-surgery on $K$} to be the closed 3-manifold
$$M(K) := \EK \cup -\EJ,$$
where $\mK \sim \mJ^{-1}$ and $\laK \sim \laJ$.  If $K$ is concordant to $J$ then $M(K)$ is homeomorphic to the boundary of the concordance.  
\end{defn}

\begin{defn}
\label{P = M(K)}
We denote $\pi_1\left( M(K) \right )$ by $P$ and define an epimorphism $\g_P: P \to G \ast_{\Z} G$ or $G$ as follows:  $P$ is a pushout and admits unique epimorphisms to $G \ast_{\Z} G$ and $G$ making the following diagram commute,
$$
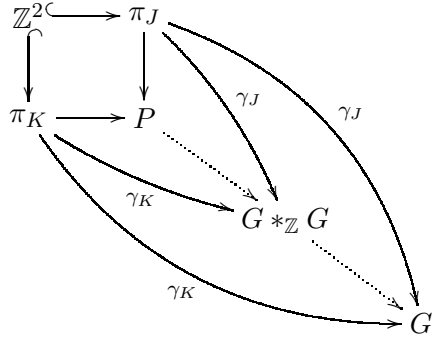
\xymatrix{
\Z^2 \ar@{^(->}[r] \ar@{^(->}[d] & \PJ \ar[d] \ar@/^/[ddr]^{\g_J} \ar@/^2pc/[dddrr]^{\g_J} \\
\PK \ar[r] \ar@/_/[drr]_{\g_K} \ar@/_2pc/[ddrrr]_{\g_K} & P \ar@{..>}[dr] & & \\
& & G \ast_{\Z} G \ar@{..>}[dr] \\
& & & G
}
$$
where the homomorphism $\gK: \PK \to G \ast_{\Z} G$ (or $\gJ: \PJ \to G \ast_{\Z} G$) maps to the first (respectively second) copy of $G$, and the epimorphism $\gK: \PK \to G$ (or $\gJ: \PJ \to G$) is induced by the inclusion of $\EK$ (respectively $\EJ$) into $M$.  We denote both $P \to G \ast_{\Z} G$ and $P \to G$ by $\g_P$.  
\end{defn}

\begin{defn}
\label{rho_n}
Let $\g_P^n: P \to \dfrac{P}{\G^{(n)} \g_P}$ be the canonical epimorphism and define
$$ \rho_n \left( M(K) \right) := \rho \left( M(K), \g_P^n \right). $$ 
\end{defn}

\begin{defn}
\label{def of N}
For a concordance $C \cong S^1 \times I$ between two knots $K$ and $L$ let
$$N = \EC \cup_{T^2 \times I} -\left( \EJ \times I \right) ,$$
where $\partial \EJ \times I \cong T^2 \times I$ is induced by the canonical framing of $\partial \EJ$, $\partial \EC \cong T^2 \times I$ has the framing induced from $\partial \EK$ (which agrees with that induced by $\partial E_L$), and $\EC$ is glued to $\EJ \times I$ by $\partial \EJ \times I \cong T^2 \times I \cong \partial \EC$. 
\end{defn}

\begin{remark}
\label{coeffs for N}
$N$ is a compact 4-manifold and
$$\partial N = M(K) \coprod -M(L) .$$
If $\pi_N = \pi_1(N)$ then $\pi_N$ is a pushout such that the following diagram commutes:
$$
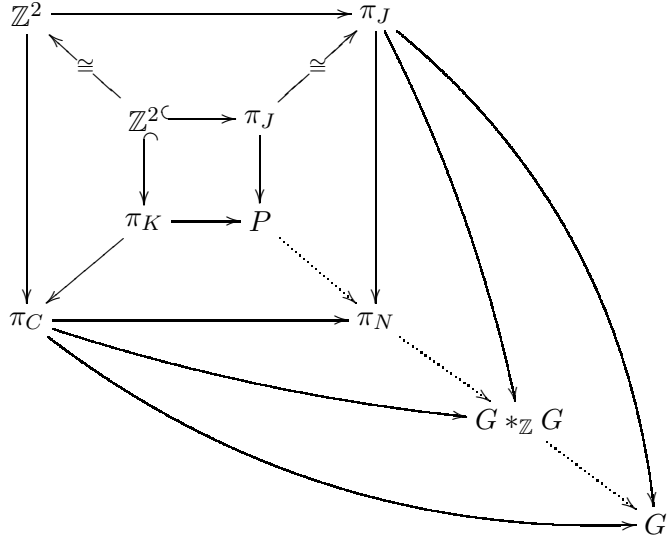
\xymatrix{
\Z^2 \ar[rrr] \ar[ddd] & & & \PJ \ar[ddd] \ar@/^/[ddddr] \ar@/^2pc/[dddddrr] \\
& \Z^2 \ar[ul]|{\cong} \ar@{^(->}[r] \ar@{^(->}[d] & \PJ \ar[d] \ar[ur]|{\cong} \\
& \PK \ar[r] \ar[dl] & P \ar@{..>}[dr] & & \\
\PC \ar[rrr] \ar@/_/[rrrrd] \ar@/_2pc/[rrrrrdd] & & & \pi_N \ar@{..>}[dr] \\
& & & & G \ast_{\Z} G \ar@{..>}[dr] \\
& & & & & G
}
$$
In this diagram $P = \pi_1\left( M(K) \right)$ as before, $\Z^2 \to \Z^2$ is the isomorphism of fundamental groups induced by $T^2 \hookrightarrow T^2 \times \{1\} \subset T^2 \times I$, $\PK \to \PC$ is induced by the inclusion of $\EK$ into $\EC$, and $\PJ \to \PJ$ is induced by $\EJ \hookrightarrow \EJ \times \{1\} \subset \EJ \times I$.  The dotted arrow from $P$ to $\pi_N$ may not be surjective but the composition of dotted arrows from $P$ to $G \ast_{\Z} G$ and $G$ are epimorphisms, as noted in Definition \ref{P = M(K)}.  In this way we have both $G$ and $G \ast_{\Z} G$-coefficient systems for $T^2$, $T^2 \times I$, $\EJ$, $\EJ \times I$, $\EK$, $\EC$, $M(K)$, and $N$.  We obtain a $G \ast_{\Z} G$-coefficient system for $M$ via the inclusion of $G$ into one of the two copies of itself in $G \ast_{\Z} G$, depending on the context.  We denote both epimorphisms $\pi_N \to G \ast_{\Z} G$ and $\pi_N \to G$ by $\g_N$. 
\end{remark}  
 
\begin{thrm}
\label{rho_n conc invt}
If $K$ and $L$ are concordant knots and $G$ is PTFA then $\rho_n \left( M(K) \right) = \rho_n \left( M(L) \right)$ for all $n$, regarded as spaces over $G$ or $G \ast_{\Z} G$.
\end{thrm}
\begin{proof}
Let $C$ be a concordance from $K$ to $L$ and let $P_K = \pi_1\left( M(K) \right)$ and $P_L = \pi_1\left( M(L) \right)$.  Take $N$ as in Definition \ref{def of N} and let $\g_N^n: \pi_N \to \dfrac{\pi_N}{\G^{(n)} \g_N}$ be the canonical epimorphism.  Our goal is to conclude that
$$\rho_n \left( M(K) \right) - \rho_n \left( M(L) \right) = \sigma^{(2)}(N, \g_N^n) - \sigma(N) = 0.$$ 
By Theorem \ref{ram}, (4) and (5) of Theorem \ref{L^2 stuff}, and Lemma \ref{quotient monomorphisms}, to get the first equality it is enough to show that $(j_K)_*$ and $(j_L)_*$, the homomorphisms induced by the inclusions of $M(K)$ and $M(L)$ into $N$, respectively, are in $\Omega^{G'}$ for $G' = G$, $G \ast_{\Z} G$.  We do this now: 

First note that $M(K)$, $M(L)$, and $N$ are finite CW-complexes, hence their fundamental groups are finitely presented.  Also, $\G \g_{P_K}$, $\G \g_{P_L}$, and $\G \g_N$ are finitely normally generated by Lemma \ref{kernels are finitely normally generated} because $\g_N$ is surjective, $\g_{P_K}$ and $\g_{P_L}$ are surjective when $G' = G$, and when $G' = G \ast_{\Z} G$ the images of $\g_{P_K}$ and $\g_{P_L}$ are the two copies of $G$ in the amalgamation.  By the commutativity of the diagram in Remark \ref{coeffs for N}, to see that $(j_K)_*$ induces a normal surjection $\G \g_{P_K} \to \G \g_N$ it is enough to show that the homomorphism induced by the inclusion $\iota: \EK \to \EC$ induces a normal surjection $\G \g_K \to \G \g_C$.  This however follows from the fact that these groups are normally generated by $\mK$ and $\iota_*(\mK)$, respectively.  Similarly, $(j_L)_*$ induces a normal surjection $\G \g_{P_L} \to \G \g_N$.  We still need to show that $(j_K)_*$ and $(j_L)_*$ induce isomorphisms on $H_1\left(- ;\Z[G'] \right)$ and epimorphisms on $H_2\left(- ;\Z[G'] \right)$ for $G'$ equal to $G \ast_{\Z} G$ and $G$.  Mayer-Vietoris sequence arguments shows that the inclusions $\EK \hookrightarrow \EC$ and $E_L \hookrightarrow \EC$ induce isomorphisms on $H_i\left(- ;\Z[G\ast_{\Z} G] \right)$, $H_i\left(- ;\Z[G] \right)$, and $H_i(- ; \Z)$ for all $i \geq 0$.  Applying the Five Lemma to the following commutative diagram of Mayer-Vietoris sequences 
$$
\xymatrix{
H_i\left( T^2; \Z[G'] \right) \ar[r] \ar[d]^{\cong} & \txt{$H_i\left( \EK; \Z[G'] \right)$\\ $\oplus$\\ $H_i\left( \EJ; \Z[G'] \right)$} \ar[r] \ar[d]^{\cong} & H_i\left( M(K); \Z[G'] \right) \ar[d] \\
H_i\left( T^2 \times I; \Z[G'] \right) \ar[r] & \txt{$H_i\left( \EC; \Z[G'] \right)$\\ $\oplus$\\ $H_i\left( \EJ \times I; \Z[G'] \right)$} \ar[r] & H_i\left( N; \Z[G'] \right) 
}
$$
$$
\xymatrix{
& \ar[r] & H_{i-1}\left( T^2; \Z[G'] \right) \ar[r] \ar[d]^{\cong} & \txt{$H_{i-1}\left( \EK; \Z[G'] \right)$\\ $\oplus$\\ $H_{i-1}\left( \EJ; \Z[G'] \right)$} \ar[d]^{\cong} \\
& \ar[r] & H_{i-1}\left( T^2 \times I; \Z[G'] \right) \ar[r] & \txt{$H_{i-1}\left( \EC; \Z[G'] \right)$\\ $\oplus$\\ $H_{i-1}\left( \EJ \times I; \Z[G'] \right)$} 
}
$$ 
where $G'$ is $G \ast_{\Z} G$, $G$, or $\{0\}$, we see that the inclusion $j_K: M(K) \hookrightarrow N$ induces an isomorphism on $H_i\left(- ;\Z[G'] \right)$ for all $i \geq 0$.  This similarly holds for $j_L: M(L) \hookrightarrow N$.  Note that for $G' = \{0\}$ this implies $\sigma(N) = 0$.  Thus, we have shown that $(j_K)_*$ and $(j_L)_*$ are in $\Omega^{G'}$ for $G' = G$, $G \ast_{\Z} G$. 

For the second equality, that $\sigma^{(2)}(N, \g_N^n) - \sigma(N) = 0$, we observed at the end of the previous paragraph that $\sigma(N) = 0$.  We therefore only need to show that $\sigma^{(2)}(N, \g_N^n) = 0$.  As noted above, $H_i \left(N, M(K); \Z \right) = 0$
for all $i \geq 0$.  Also, $\dfrac{\pi_N}{\G^{(n)} \g_N}$ is PTFA by Lemma \ref{loc is PTFA}.  It follows from \cite{cot} (Proposition 2.10) that 
$$H_i \left( N, M(K); \mathcal{K} \left( \dfrac{\pi_N}{\G^{(n)} \g_N} \right) \right) = 0$$
for all $i \geq 0$, where $\mathcal{K} \left( \dfrac{\pi_N}{\G^{(n)} \g_N} \right)$ is the classical right ring of quotients of $\Z \left[ \dfrac{\pi_N}{\G^{(n)} \g_N} \right]$ with respect to all non-zero elements.  As $\mathcal{U} \left( \dfrac{\pi_N}{\G^{(n)} \g_N} \right)$ is a flat $\mathcal{K} \left( \dfrac{\pi_N}{\G^{(n)} \g_N} \right)$-module, 
$$H_i \left( N, M(K); \mathcal{U} \left( \dfrac{\pi_N}{\G^{(n)} \g_N} \right) \right) = 0.$$
Hence, $\sigma^{(2)}(N, \g_N^n) = 0$. 
\end{proof}

\subsection{Constructing and distinguishing characteristic knots}

We now construct examples of $J$-characteristic knots that are distinguished, up to concordance, by their $\rho$-invariants.  To achieve this we follow the general construction of Harvey in \cite{harvey2}, replacing the Harvey series with the rational $G$-local derived series.  We assume as always that the manifold $M$ has PTFA fundamental group.  We continue to work in the smooth category. \\

Let $K$ be a knot in $M$, let $\eta \subset \EK$ be an embedded curve that bounds an embedded disk in $M$, and let $N(\eta)$ be a regular neighborhood of $\eta$ in the interior of $\EK$.  Let $L$ be a knot in $S^3$.  Then 
$$\left( M - N(\eta) \right) \cup -\left( S^3 - N(L) \right), $$
where $\mu_L \sim \lambda_{\eta}^{-1}$ and $\lambda_L \sim \mu_{\eta}$, is diffeomorphic to $M$ via a diffeomorphism that is the identity outside a regular neighborhood of the disk bounded by $\eta$.  In particular, this diffeomorphism induces the identity homomorphism on $G$.  

\begin{defn}
Define $K(\eta, L)$ to be the image of $K$ in $\left( M - N(\eta) \right) \cup -\left( S^3 - N(L) \right) \cong M$.  That is, the image of $K$ under the inclusion
$$K \hookrightarrow M-N(\eta) \hookrightarrow \left( M - N(\eta) \right) \cup -\left( S^3 - N(L) \right).$$
Denote $J$-surgery on $K(\eta, L)$ by $M(K, \eta, L)$, so, as in Definition \ref{J surgery}, 
$$M(K, \eta, L) = E_{K(\eta, L)} \cup_{T^2} -\EJ .$$  
\end{defn}

\begin{defn}
Define $Q = \pi_1(M(K, \eta, L))$ and let $\g_Q: Q \rightarrow G$ be the epimorphism to $G \ast_{\Z} G$ or $G$, as in Definition \ref{P = M(K)}.
\end{defn}

\begin{remark} 
\label{sat const}$\phantom{}$

\begin{enumerate}
\item[(i)] The knot $K(\eta, L)$ is obtained in $M$ by grasping all the strands of $K$ that pass through the disk bounded by $\eta$ and tying them in the (untwisted) knot $L$, and is therefore homotopic to $K$.  
\item[(ii)] The exterior of $K(\eta, L)$ is $E_{K(\eta, L)} = \left( \EK - N(\eta) \right) \cup -\left( S^3 - N(L) \right)$. 
\end{enumerate}
\end{remark}

\begin{defn}
For future use we define 
$$Z_L := \left(S^3 - N(L) \right) \cup -ST$$ 
to be the usual 0-surgery on $L$ in $S^3$. 
\end{defn}

Following Harvey in \cite{harvey2}, construct a cobordism $C$ between $M(K)$ and $M(K, \eta, L)$ as follows:  let $W$ be a compact, oriented 4-manifold with $\partial W = Z_L$, $\sigma(W) = 0$, and $\pi_1(W) = \Z$ generated by $\iota_*(\mu_L)$, where $\iota: Z_L \hookrightarrow W$ is inclusion of the boundary.  Note that $H_2 \left( \pi_1(W) \right) = H_2(\Z) = 0$, so $H_2(W) = \text{image}\{\pi_2(W) \rightarrow H_2(W)\}$ by Hopf's Theorem.  

\begin{defn} 
\label{C} 
Define $C$ by 
$$C := \left(M(K)\times I \right) \cup -W, $$ 
where $N(L) \subset W$ is glued to $N(\eta) \times \{1\} \subset M(K) \times \{1\}$ by $\lambda_L \sim \mu_{\eta}\times \{1\}$ and $\mu_L \sim \lambda_{\eta}^{-1}\times \{1\}$.  
\end{defn}

\begin{remark}
\label{remark C}
We constructed $C$ so that
$$\partial C = -M(K) \coprod M(K, \eta, L), $$
making $C$ a cobordism between $M(K)$ and $M(K, \eta, L)$.  The group $\pi_1(C)$ is a pushout and there are uniquely defined epimorphisms from $\pi_1(C)$ to $G \ast_{\Z} G$ and $G$ satisfying the following diagram, 
$$
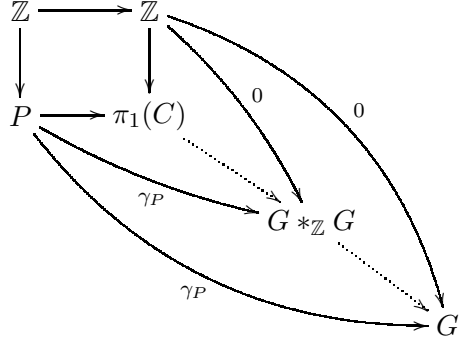
\xymatrix{
\Z \ar[r] \ar[d] & \Z \ar[d] \ar@/^/[ddr]^0 \ar@/^2pc/[dddrr]^0 \\
P \ar[r] \ar@/_/[drr]_{\g_P} \ar@/_2pc/[ddrrr]_{\g_P} & \pi_1(C) \ar@{..>}[dr]\\
& & G \ast_{\Z} G \ar@{..>}[dr] \\
& & & G
}
$$ 
where $P = \pi_1\left( M(K) \right) = \pi_1\left( M(K) \times I \right)$, the $\Z$ on the upper-left is $\pi_1\left( N(\eta) \right)$, and the $\Z$ on the upper-right is $\pi_1(W)$.  Note that the homomorphisms induced by the inclusions $\partial N(\eta) \hookrightarrow N(\eta)$ and $Z_L \hookrightarrow W$ induce the trivial $G$ and $G \ast_{\Z} G$-coefficient systems on $\partial N(\eta)$ and $Z_L$, respectively. 
\end{remark}

\begin{defn}
Define $E = \pi_1(C)$ and define $\g_E: E \to G \ast_{\Z} G$ or $G$ to be the epimorphisms in Remark \ref{remark C}.
\end{defn}

\begin{defn}
\label{i and j}
Define $i$ and $j$ to be the inclusions of $M(K)$ and $M(K, \eta, L)$ into $\partial C$, respectively. 
\end{defn}

\begin{remark}
\label{i, j over M}
The homomorphisms induced by $i$ and $j$ on fundamental groups, respectively $i_*: P = \pi_1\left( M(K) \right) \to E$ and $j_*: Q = \pi_1 \left( M(K, \eta, L) \right) \to E$, make the following diagram commute:
$$
\xymatrix{
P \ar[r] ^{i_*} \ar[dr]_{\g_P} \ar@/_/[ddr]_{\g_P} & E \ar[d]|{\g_E} & Q \ar[l]_{j_*} \ar[dl]^{\g_Q} \ar@/^/[ddl]^{\g_Q} \\
& G \ast_{\Z} G \ar[d] \\
& G
}
$$
\end{remark}

\begin{remark}
\label{H-results}
Harvey showed in \cite{harvey2} that 
\begin{enumerate}
\item $i_*: \pi_1(M(K)) \rightarrow \pi_1(C)$ is an isomorphism (Lemma 5.4),
\item $j_*:  \pi_1(M(K, \eta, L)) \rightarrow \pi_1(C)$ is an epimorphism (Lemma 5.5), and
\item $j_*: \pi_1(M(K, \eta, L)) \rightarrow \pi_1(C)$ induces an isomorphism on $H_1( - ; \Z)$ and an epimorphism on $H_2( - ; \Z)$ (Lemma 5.6). 
\end{enumerate}
\end{remark}

\begin{lem}
\label{H_1 iso and H_2 epi}
The epimorphism $j_*: \pi_1(M(K, \eta, L)) \rightarrow \pi_1(C)$ induces an isomorphism on $H_1( - ; \Z[G'])$ and an epimorphism on $H_2( - ; \Z[G'])$ for $G' = G$ and $G \ast_{\Z} G$.
\end{lem}
\begin{proof} 
Recall that $E = \pi_1(C)$, $P = \pi_1\left( M(K) \right)$, and $Q = \pi_1\left( M(K, \eta, L) \right)$.  Let $R = \pi_1\left( M(K) - N(\eta) \right)$ and let $p: M(K) - N(\eta) \hookrightarrow M(K)$ and $q: M(K) - N(\eta) \hookrightarrow M(K, \eta, L)$ be the inclusions.  We first show that $j_*$ induces an isomorphism on $H_1( - ; \Z[G'])$, and we note that this part of the proof essentially follows the proof of Lemma 5.6 of \cite{harvey2}.  As was noted in Remarks \ref{H-results} and \ref{i, j over M}, $j$ induces an epimorphism on fundamental groups such that $\g_E \circ j_* = \g_Q$.  Hence
$$\xymatrix{ H_1\left( Q; \Z[G'] \right) \cong \dfrac{\G \g_Q}{\G_2 \g_Q} \ar[r]^{j_*} & \dfrac{\G \g_E}{\G_2 \g_E} \cong H_1\left( E; \Z[G'] \right) }$$
is an epimorphism.  To see that $j_*$ is injective, consider the following commutative diagram:
$$
\xymatrix{
H_1\left( M(K) - N(\eta); \Z[G'] \right) \ar[r]^(.575){p_*} \ar[d]_{q_*} & H_1\left( M(K); \Z[G'] \right) \ar[d]^{i_*}_{\cong} \\
H_1\left( M(K, \eta, L); \Z[G'] \right) \ar[r]^(.6){j_*} & H_1\left( C; \Z[G'] \right)
}
$$ 
In the Mayer-Vietoris sequence
$$\xymatrix{ 
H_1\left( \partial N(\eta); \Z[G'] \right) \ar[r] & \txt{$H_1\left( M(K) - N(\eta); \Z[G'] \right)$ \\ $\oplus$ \\ $H_1\left( S^3 - N(L); \Z[G'] \right)$} \ar[r] & H_1\left( M(K, \eta, L); \Z[G'] \right) \\
\txt{$\phantom{MMMM}$} \ar[r] & H_0\left( \partial N(\eta); \Z[G'] \right) \ar[r] & \txt{$H_0\left( M(K) - N(\eta); \Z[G'] \right)$ \\ $\oplus$ \\ $H_0\left( S^3 - N(L); \Z[G'] \right)$} 
}
$$
the homomorphism $H_i\left( \partial N(\eta); \Z[G'] \right) \to H_i\left( S^3 - N(L); \Z[G'] \right)$ is an epimorphism for $i=1$ and an isomorphism for $i=0$.  Hence, $q_*$ is surjective.  Also, $p_*$ in the above diagram is surjective because $p$ induces an epimorphism on fundamental groups over $G'$.  Suppose that $x \in H_1\left( M(K, \eta, L); \Z[G'] \right)$ and that $j_*(x) = e$.  There is an $\tilde{x} \in H_1\left( M(K) - N(\eta); \Z[G'] \right)$ such that $q_*(\tilde{x}) = x$.  As $i_* \circ p_*(\tilde{x}) = j_* \circ q_*(\tilde{x}) = e$ and $i_*$ is an isomorphism, $\tilde{x}$ is in the kernel of $p_*$.  The kernel of $p_*: R \to P$ is normally generated by $\mu_{\eta}$, and $\mu_{\eta} \in \G \g_R$.  It follows that $\tilde{x}$ is a product of conjugates of $\mu_{\eta}$.  Since $q_*(\mu_{\eta}) = \lambda_L = e$ in $H_1\left( M(K, \eta, L); \Z[G'] \right)$, $x$ is trivial.  It follows that $j_*$ induces an isomorphism on $H_1( - ; \Z[G'])$. 

We still need to show that $j_*$ induces an epimorphism on $H_2( - ; \Z[G'])$.  We have 
$$\xymatrix{ H_2(M(K); \Z[G']) \ar[r] & H_2(P; \Z[G']) \ar[r]^{i_*} & H_2(E; \Z[G']) }$$ 
where the first homomorphism is surjective by Hopf's Theorem and the second is an isomorphism by Remarks \ref{H-results} and \ref{i, j over M}.  We will show that $H_2(M(K); \Z[G']) = 0$, implying that $H_2(E; \Z[G']) = 0$ and, hence, that $j_*$ induces an epimorphism on $H_2(-; \Z[G'])$.  Consider the following Mayer-Vietoris sequence associated to $M(K)$: 
$$\xymatrix{ \txt{$H_2(\EK; \Z[G'])$ \\ $\oplus$ \\ $H_2(\EJ; \Z[G'])$ } \ar[r] & H_2(M(K); \Z[G']) \ar[r] & H_1(T^2; \Z[G']) \ar[r] & \txt{$H_1(\EK; \Z[G'])$ \\ $\oplus$ \\ $H_1(\EJ; \Z[G'])$ }. }$$ 
When $G' = G$ it follows from Proposition \ref{basics of \EK} that $H_2(\EK; \Z[G]) = H_2(\EJ; \Z[G]) = 0$ and $H_1(T^2; \Z[G]) \rightarrow H_1(\EK; \Z[G]) \oplus H_1(\EJ; \Z[G])$ is injective.  If $G' = G \ast_{\Z} G$ then the coefficient systems on $\EK$ and $\EJ$ have the left and right copies of $G$ in the amalgamation $G \ast _{\Z} G$ as their respective images.  Hence, we still have that $H_2(\EK; \Z[G \ast _{\Z} G]) = H_2(\EJ; \Z[G \ast _{\Z} G]) = 0$ and 
$$H_1\left( T^2; \Z[G \ast _{\Z} G] \right) \rightarrow H_1\left( \EK; \Z[G \ast _{\Z} G] \right) \oplus H_1\left( \EJ; \Z[G \ast _{\Z} G] \right)$$
is injective.  It follows that $H_2(M(K); \Z[G']) = 0$ when $G' = G$ and $G \ast _{\Z} G$, as desired.
\end{proof}

\begin{lem}
\label{i, j in Omega}
$i_*: \pi_1(M(K)) \rightarrow \pi_1(C)$ and $j_*: \pi_1(M(K, \eta, L) \rightarrow \pi_1(C)$ are in $\Omega^{G}$ and $\Omega^{G \ast _{\Z} G}$.
\end{lem}
\begin{proof}
Let $G'$ denote $G$ or $G \ast _{\Z} G$.  Recall that $E = \pi_1(C)$, $P = \pi_1\left( M(K) \right)$, and $Q = \pi_1\left( M(K, \eta, L) \right)$.  Since $M(K)$, $M(K, \eta, L)$, and $C$ are finite CW-complexes, $P$, $Q$, and $E$ are finitely presented groups.  Both $i_*$ and $j_*$ are surjective morphisms in $\mathcal{G}^{G'}$ by Remarks \ref{H-results} and \ref{i, j over M}, and therefore induce epimorphisms $\G \g_P \to \G \g_E$ and $\G \g_Q \to \G \g_E$, respectively.  Also, they induce isomorphisms on $H_1(-; \Z[G'])$ and epimorphisms on $H_2(-; \Z[G'])$ by Remark \ref{H-results} and Lemma \ref{H_1 iso and H_2 epi}, respectively.  We therefore only need to show that the kernels of $\g_P$, $\g_Q$, and $\g_E$ are finitely normally generated in $P$, $Q$, and $E$, respectively, but this follows from Lemma \ref{kernels are finitely normally generated}.
\end{proof}

\begin{defn}
Recall that $E = \pi_1(C)$.  Define 
$$\tau: \pi_1(Z_L) \to  E$$
to be the homomorphism induced by inclusion of $Z_L$ into $C$ and define
$$\tau_i : \pi_1(Z_L) \rightarrow \dfrac{E}{\G^{(i)} \g_E}$$
to be the composition of $\tau$ with the epimorphism $E \rightarrow \dfrac{E}{\G^{(i)} \g_E}$. 
\end{defn}

The following lemma is our analogue of Lemma 5.7 in \cite{harvey2} and the proof is almost identical.
\begin{lem}
\label{tau}
Suppose that $\eta \in \DPn$ and $\eta \notin \DPnn$, where $P = \pi_1(M(K))$ as always.  Then
\begin{equation*}
\text{Image}\left( \tau_i \right) = \left\{
\begin{array}{ccl}
\{e\} & \text{if} & 0\leq i \leq n \text{, and}\\
\Z & \text{if} & i = n+1. 
\end{array} \right.
\end{equation*} 
\end{lem}
\begin{proof} 
Recall that $P = \pi_1(M(K))$, $E = \pi_1(C)$, and that $i_*:P \rightarrow E$ is induced by inclusion of $M(K)$ into $C$.  It is well known that $\pi_1(Z_L)$ is normally generated by the meridian $\mu_L$.  In $E$, $\tau(\mu_L)$ is identified with $\lambda_{\eta}^{-1} \times \{1\}$.  As $\lambda_{\eta} \in \DPn$ and $i_* \left( \G^{(n)} \g_P \right) \subset \G^{(n)} \g_E$ by Lemma \ref{derived to derived}, $\tau \left( \pi_1(Z_L) \right) \subset \G^{(n)} \g_E$.  Hence, $\tau_i \left( \pi_1(Z_L) \right) = \{0\}$ for $i \leq n$. 

We still need to show that the image of $\tau_{n+1}$ is $\Z$.  By Lemma \ref{quotient monomorphisms} and Lemma \ref{i, j in Omega}, $i_*$ induces a monomorphism 
$$ \overline{i_*}: \dfrac{P}{\DPi} \to \dfrac{E}{\DEi} $$ 
for all $i$.  Moreover, 
$$\tau \left( [\pi_1(Z_L), \pi_1(Z_L)] \right) \subset [\DEn, \DEn] \subset \DEnn ,$$
so the homomorphism $\tau_{n+1}: \pi_1(Z_L) \to \dfrac{E}{\DEnn} $ factors through $\dfrac{\pi_1(Z_L)}{[\pi_1(Z_L), \pi_1(Z_L)]} \cong \Z$.  As above, $\tau(\mu_L) = \lambda_{\eta}^{-1} \times \{1\}$ in $E$.  We chose $\eta \notin \DPnn$, so $\tau_{n+1}(\mu_L) \in \dfrac{E}{\DEnn}$ is nontrivial.  Because $\dfrac{E}{\DEnn}$ is torsion-free by construction, it follows that $\text{Image}(\tau_{n+1}) = \Z$.
\end{proof}

\begin{lem}
\label{rho of Z_L}
\begin{equation*}
\rho \left(Z_L, \tau_i \right) = \left\{
\begin{array}{ccl}
0 & \text{if} & 0\leq i \leq n\\
\int_{S^1} \sigma_{\omega}(L) d\omega & \text{if} & i = n+1 \\
\end{array} \right. ,
\end{equation*}
where $\sigma_{\omega}(L)$ is the Levine-Tristram signature of $L$ at $\omega \in S^1$.
\end{lem}
\begin{proof}
If $0 \leq i \leq n$ then the image of $\tau_i$ is trivial by Lemma \ref{tau}.  Hence, $\rho \left(Z_L, \tau_i \right) = 0$ by (2) of Theorem \ref{L^2 stuff}.  The image of $\tau_{n+1}$ is $\Z$ by Lemma \ref{tau}, so $\rho \left(Z_L, \tau_i \right) = \int_{S^1} \sigma_{\omega}(L) d\omega$ by Lemma 5.3 of \cite{cot2}.
\end{proof}

The following theorem is our analogue of Theorem 5.8 in \cite{harvey2} and, again, our proof is almost identical to the one presented there.
\begin{thrm}
\label{rhos = LT}
Suppose that $\eta \in \DPn$ and $\eta \notin \DPnn$, where $P = \pi_1(M(K))$ as always.  Then for $G' = G$ and $G \ast_{\Z} G$,
\begin{equation*}
\rho_i\left(M(K, \eta, L)\right) - \rho_i\left(M(K)\right) = \left\{
\begin{array}{ccl}
0 & \text{if} & 0\leq i \leq n \text{, and} \\
\int_{S^1} \sigma_{\omega}(L) d\omega & \text{if} & i = n+1.
\end{array} \right.
\end{equation*}
\end{thrm}
\begin{proof}
We will create a space $C_V$ from $C$ such that 
$$\partial C_V = -M(K) \coprod M(K, \eta, L) \coprod -Z_L$$
and all $H_2(C_V)$, with relevant coefficients, comes from the boundary.

Let $V$ be a neighborhood of $Z_L = \partial W$ in $W$ homeomorphic to $Z_L \times I$ and define
$$C_V := \overline{ C - \left( W - V \right) } . $$
Then $\partial C_V = -M(K) \coprod M(K, \eta, L) \coprod -Z_L$ as desired, and the inclusion of $C_V$ into $C$ induces an epimorphism on fundamental groups.  Recall that $P = \pi_1(M(K))$, $Q = \pi_1(M(K, \eta, L))$, and $E = \pi_1(C)$.  Recall also that for $\Lambda = P, Q, E$, $\g_{\Lambda}$ factors through the canonical epimorphism $\g_{\Lambda}^i: \Lambda \to \dfrac{\Lambda}{\G^{(i)} \g_{\Lambda}}$.  The homomorphisms $\overline{i_*}: \dfrac{P}{\DPi} \to \dfrac{E}{\DEi}$ and $\overline{ j_*}: \dfrac{Q}{\DQi} \to \dfrac{E}{\DEi}$
are injective for all $i \geq 0$ by Lemmas \ref{quotient monomorphisms} and \ref{i, j in Omega}.  By (4) of Theorem \ref{L^2 stuff}, 
$$\rho \left( M(K), \overline{i_*} \circ \g_P^i \right) = \rho_i(M(K))$$
and
$$\rho \left( M(K, \eta, L),  \overline{j_*} \circ \g_Q^i \right) = \rho_i (M(K, \eta, L)) $$
for all $i \geq 0$.  Also, 
\begin{equation*}
\rho \left(Z_L, \tau_i \right) = \left\{
\begin{array}{ccl}
0 & \text{if} & 0\leq i \leq n\\
\int_{S^1} \sigma_{\omega}(L) d\omega & \text{if} & i = n+1
\end{array} \right.
\end{equation*}
by Lemma \ref{rho of Z_L}.  If $\g_V^i: \pi_1(C_V) \to \dfrac{E}{\G^{(i)} \g_E}$ is the composition
$$\xymatrix{\pi_1(C_V) \ar[r] & E \ar[r]^{} & \dfrac{E}{\G^{(i)} \g_E} ,}$$
where the first homomorphism is induced by the inclusion of $C_V$ into $C$, then, by Theorem \ref{ram},
$$\rho_i\left(M(K, \eta, L)\right) - \rho_i\left(M(K)\right) = \sigma^{(2)} \left(C_V, \g_V^i \right) - \sigma \left(C_V\right) $$
for $i \leq n$ and
$$\rho_{n+1}\left(M(K, \eta, L)\right) - \rho_{n+1}\left(M(K)\right) - \int_{S^1} \sigma_{\omega}(L) d\omega = \sigma^{(2)} \left(C_V, \g_V^{n+1} \right) - \sigma \left(C_V\right). $$

For the sake of the following argument we make the convention that $\G^{(-1)} \g_E = E$, so 
$$\g_V^{-1}: \pi_1(C_V) \to \dfrac{E}{\G^{(-1)} \g_E} = \{0\}$$ 
is the trivial homomorphism.  We now show that $\sigma^{(2)} \left(C_V, \g_V^i \right) = 0$ for all $i \geq -1$ (when $i = -1$ this shows that $\sigma \left(C_V\right) = 0$) by showing that 
$$H_2\left( \partial C_V; \Z \left[ \dfrac{E}{\DEi} \right] \right) \to H_2\left( C_V; \Z \left[ \dfrac{E}{\DEi} \right] \right) $$
is surjective for $i \geq -1$.  Decomposing $C_V$ as
$$C_V = \left( M(K)\times I \right) \bigcup_{N(\eta) \times \{1\}} \left( Z_L \times I \right) $$
gives the following Mayer-Vietoris sequence:
$$
\xymatrix{
\txt{$H_2 \left( M(K)\times I; \Z \left[\dfrac{E}{\DEi} \right] \right)$\\ $\oplus$\\ $H_2 \left( Z_L\times I; \Z \left[\dfrac{E}{\DEi} \right] \right)$}  \ar[r] & H_2 \left( C_V; \Z \left[\dfrac{E}{\DEi} \right] \right) \ar[r] & \text{$\phantom{MMMMMMM}$}
}
$$
$$
\xymatrix{
H_1 \left( N(\eta); \Z \left[\dfrac{E}{\DEi} \right] \right) \ar[r] & \txt{$H_1 \left( M(K)\times I; \Z \left[\dfrac{E}{\DEi} \right] \right)$\\ $\oplus$\\ $H_1 \left( Z_L\times I; \Z \left[\dfrac{E}{\DEi} \right] \right)$} .
}
$$
If $i \leq n$ then $\pi_1(N(\eta))$, which is generated by $\lambda_{\eta}$, maps trivially to $\dfrac{E}{\DEi}$.  Also, the image of $\pi_1(Z_L)$ is trivial in $\dfrac{E}{\DEi}$ for $i \leq n$ by Lemma \ref{tau}.  As $\mu_L$ is identified with $\lambda_{\eta}^{-1}\times \{1\}$ in $C_V$,
$$H_1 \left( N(\eta); \Z \left[\dfrac{E}{\DEi} \right] \right) \to H_1 \left( Z_L\times I; \Z \left[\dfrac{E}{\DEi} \right] \right) $$
is an isomorphism.  It follows that
$$H_2 \left( M(K)\times I; \Z \left[\dfrac{E}{\DEi} \right] \right) \oplus H_2 \left( Z_L\times I; \Z \left[\dfrac{E}{\DEi} \right] \right)  \to H_2 \left( C_V; \Z \left[\dfrac{E}{\DEi} \right] \right) $$
is surjective for $-1 \leq i \leq n$.  For $i=n+1$, the image of $\pi_1(N(\eta))$ in $\dfrac{E}{\DEi}$ is $\Z$.  Hence, $H_1 \left( N(\eta); \Z \left[\dfrac{E}{\DEi} \right] \right) = 0$ and 
$$\xymatrix{ 
\txt{$H_2 \left( M(K)\times I; \Z \left[\dfrac{E}{\DEnn} \right] \right)$\\ $\oplus$\\ $H_2 \left( Z_L\times I; \Z \left[\dfrac{E}{\DEnn} \right] \right)$}  \ar[r] & H_2 \left( C_V; \Z \left[\dfrac{E}{\DEnn} \right] \right) 
}$$
is surjective.  This shows that $H_2\left( \partial C_V; \Z \left[ \dfrac{E}{\DEi} \right] \right) \to H_2\left( C_V; \Z \left[ \dfrac{E}{\DEi} \right] \right)$ is surjective, as desired.  It now follows from (1) of Theorem \ref{L^2 stuff} that $\sigma^{(2)} \left(C_V, \g_V^i \right) = \sigma \left(C_V\right) = 0$ for $0 \leq i \leq n+1$.
\end{proof}

The following theorem is analogous to Theorem 5.11 in \cite{harvey2}.  However, our objective is to construct a family of knots, whereas the objective in \cite{harvey2} is to obtain a family of 3-manifolds.
\begin{thrm}
\label{dense subset}
Let $G'$ be $G$ or $G \ast_{\Z} G$.  Let $K$ be a knot and suppose that $\eta \in \DKn - \DKnn$ bounds an embedded disk in $M$.  Then there is a countably infinite subset of knots $\{K_{\epsilon}\}$ homotopic to $K$ in $M$ such that $\rho_{i}(M(K_{\epsilon})) = \rho_{i}(M(K))$ for $i \leq n$ and all $\epsilon$, and $\{ \rho_{n+1}(M(K_{\epsilon})) \}$ is a dense subset of $\R$.
\end{thrm}
\begin{proof}
In \cite{LivCha} Cha and Livingston show that there is a family of knots $L_{\epsilon}$ in $S^3$ such that $\int_{S^1} \sigma_{\omega}(L_{\epsilon}) d\omega$ is a dense subset of $(-2, 2)$.  The Levine-Tristram signature is additive under connect sum of knots, so connect-summing these gives rise to a family of knots, which we will also call $L_{\epsilon}$, such that $\int_{S^1} \sigma_{\omega}(L_{\epsilon}) d\omega$ is a dense subset of $\R$.  It was observed in Remark \ref{sat const} that the knot $K_{\epsilon} = K(\eta, L_{\epsilon})$ is homotopic to $K$.  Hence, the result now follows by applying Theorem \ref{rhos = LT} to the family of knots $K_{\epsilon}$ in $M$.
\end{proof}

For the next theorem, recall that a knot $K$ is {\em $J$-characteristic} if there is a continuous map $\alpha: M \to M$ such that $\alpha(K) = J$ and $\alpha(M-K) \subseteq M-J$.

\begin{thrm}
\label{dense J}
Let $G'$ be $G$ or $G \ast_{\Z} G$ and suppose that $\eta \in \DJn - \DJnn$ bounds an embedded disk in $M$.  There are knots $K_{\epsilon}$ as in Theorem \ref{dense subset} such that 
\begin{enumerate}
\item[(i)] Each $K_{\epsilon}$ is $J$-characteristic,
\item[(ii)] For each $K_{\epsilon}$, $\rho_{i}(M(K_{\epsilon})) = \rho_{i}(M(J))$ for $i \leq n$, and 
\item[(iii)] $\{ \rho_{n+1}(M(K_{\epsilon})) \}$ is a dense subset of $\R$.
\end{enumerate}
\end{thrm}
\begin{proof}
Let $\{ L_{\epsilon} \}$ be a family of knots in $S^3$ whose Levine-Tristram signatures form a dense subset of $\R$, as in the proof of Theorem \ref{dense subset}, and define $K_{\epsilon} = J(\eta, L_{\epsilon})$.  Then (ii) and (iii) follow from Theorem \ref{dense subset}.  For (i) we must construct for each $\epsilon$ a map $\alpha_{\epsilon}: M \to M$ such that $\alpha_{\epsilon} (J(\eta, L_{\epsilon})) = J$ and $\alpha_{\epsilon} (M - {J(\eta, L_{\epsilon})}) \subseteq M - J$.  Recall that 
$$E_{J(\eta, L_{\epsilon})} = \left( \EJ - N(\eta) \right) \cup -\left( S^3 - N(L_{\epsilon}) \right), $$ 
where $\mu_{L_{\epsilon}} \sim \lambda_{\eta}^{-1}$ and $\lambda_{L_{\epsilon}} \sim \mu_{\eta}$.  The abelianization homomorphism $\pi_1\left( S^3 - N(L_{\epsilon}) \right) \to \Z$ induces a map 
$$\alpha_{\epsilon}^{''}: S^3 - N(L_{\epsilon}) \to S^3 - N(U),$$
where $U$ is the unknot in $S^3$, that restricts to a homeomorphism of the boundaries with $\alpha_{\epsilon}^{''}( \mu_{L_{\epsilon}}) = \mu_U$.
We may extend $\alpha_{\epsilon}^{''}$ to a map 
$$\alpha_{\epsilon}^{'}: E_{J(\eta, L_{\epsilon})} \to \EJ = \left( \EJ - N(\eta) \right) \cup -\left( S^3 - N(U) \right)$$
by defining $\alpha_{\epsilon}^{'}$ to be the identity on $\EJ - N(\eta)$.  Finally, as $\alpha_{\epsilon}^{'}$ maps $\partial E_{J(\eta, L_{\epsilon})}$ homeomorphically to $\partial \EJ$ with $\alpha_{\epsilon}^{'}(\mu_{J(\eta, L_{\epsilon})}) = \mJ$, we may extend $\alpha_{\epsilon}^{'}$ to a map 
$$\alpha_{\epsilon}: M \to M$$
by identifying the first $M$ with $E_{J(\eta, L_{\epsilon})} \cup ST$ and the second $M$ with $\EJ \cup ST$, and where the first solid torus is attached along $\partial E_{J(\eta, L_{\epsilon})}$ so that $\mu_{J(\eta, L_{\epsilon})}$ bounds a disk and the second solid torus is attached along $\partial \EJ$ so that $\mJ$ bounds a disk.  In $M$, the core of the solid torus attached to $E_{J(\eta, L_{\epsilon})}$ is $J(\eta, L_{\epsilon})$ and the core of the solid torus attached to $\EJ$ is $J$.  Because $\alpha_{\epsilon}\left(E_{J(\eta, L_{\epsilon})}\right) = \EJ$, we can isotope $\alpha_{\epsilon}$ so that $\alpha_{\epsilon} \left( J(\eta, L_{\epsilon}) \right) = J$ and $\alpha_{\epsilon} \left( M - J(\eta, L_{\epsilon}) \right) \subseteq M - J$.  Hence, each $K_{\epsilon} = J(\eta, L_{\epsilon})$ is $J$-characteristic.
\end{proof}

\section{Construction of the $G$-localization}
\label{const of localization}

We now construct the localization $(E, p)$ from Theorem \ref{M-localization}.  The construction presented here follows Levine's construction of algebraic closure of groups \cite{alg.closureII} and may be known to some experts although, to our knowledge, has not appeared in the literature.  We note that K. Orr and J. C. Cha use the existence of this localization in a forthcoming paper \cite{orr-cha}, although they do not include a construction. \\

\begin{defn}
A {\em system of equations} over $\g_A: A \to \M$ is a finite collection of equations
$$ \left\{ x_i = w_i(x_1, \ldots, x_n) \right\}_{i=1}^n $$
where $w_i \in \text{Ker} \left\{ A \ast F(x_1, \ldots, x_n) \rightarrow \M \ast F(x_1, \ldots, x_n) \right\}$. \\

A {\em solution} to the system is a set $\{g_i\}_{i=1}^n \subset A$ such that setting $x_i = g_i$ gives a true equation in $A$.  If $\{g_i\}_{i=1}^n$ is a solution then $g_i \in \G \g_A$ and $g_iw_i^{-1}$ is in the kernel of the homomorphism
$$A \ast F(x_1, \ldots, x_n) \rightarrow A$$
that maps
$$x_i \mapsto g_i$$
for $i = 1, \ldots, n$. \\
\end{defn}

\begin{defn}
$\g_X$ is {\em algebraically closed}, written $\g_X \in AC_G$, if every system of equations over $\g_X$ has a unique solution in $X$.  An {\em algebraic closure} of $\g_A$ is a morphism $\g_A \to \g_{\hat{A}}$ in $\mathcal{G}^G$ with $\g_{\hat{A}} \in AC_G$ satisfying the universal property that if $f: \g_A \to \g_X$ with $\g_X \in AC_G$ then there is a unique morphism $\hat{f}: \g_{\hat{A}} \to \g_X$ making the following diagram commute: 
$$
\xymatrix{
A \ar[dr]_f \ar[rr] & & \hat{A} \ar[dl]^{\hat{f}} \\
& X
}
$$
\end{defn}

$\phantom{}$ \\

\begin{defn}
A {\em $\Pi$-perfect subgroup of $\g_A$} is a normal subgroup $N$ of $A$ such that
\begin{enumerate}
\item[(i)] $[N, \G \g_A] = N$ (recall, $\G \g_A = \text{Ker}\{ \g_A \}$) and 
\item[(ii)] $N$ is finitely normally generated in $A$.  
\end{enumerate} 
Note that if $N$ is $\Pi$-perfect then $N \unlhd \G \g_A$. \\
\end{defn}

\begin{thrm}
\label{AC has no Pi-perfect subgps}
If $\g_X \in AC_G$ and $N$ is a $\Pi$-perfect subgroup of $\g_X$ then $N = \{e\}$.
\end{thrm}
\begin{proof}
Choose a finite normal generating set $\{g_1, \ldots g_n\}$ for $N$.  Since $N = [N, \G \g_X]$ we may write
$$ g_i = w_i(g_1, \ldots g_n) := \prod_k a_k [g_{j_k}, b_k]^{\pm}a_k^{-1} $$
where $a_k \in X$, $b_k \in \G \g_X$, and the product is finite.  Replace $g_i$ with $x_i$ to get words
$$ x_i = w_i(x_1, \ldots x_n) $$
in $\text{Ker}\left\{ X\ast F(x_1, \ldots x_n) \rightarrow G \ast F(x_1, \ldots, x_n) \right\}$.  There is a unique solution to the system of equations $\{x_i = w_i\}_{i = 1}^n$ because $\g_X \in AC_G$.  Both $\{x_i = e\}_{i=1}^n$ and $\{x_i = g_i\}_{i=1}^n$ are solutions, so $g_i = e$ for all $i$. 
\end{proof}

$\phantom{}$ \\

\begin{thrm}
\label{Pi-perfect iff nonunique solutions}
There exists a nontrivial $\Pi$-perfect subgroup $N$ of $\g_A$ if and only if there exists a system of equations
$$ \{ x_i = w_i \}_{i=1}^n $$
over $\g_A$ with more than one solution.
\end{thrm}
\begin{proof} 
Suppose that $N$ is a nontrivial $\Pi$-perfect subgroup and let $\{g_i\}_{i=1}^n$ be a normal generating set for $N$.  As in the proof of Theorem \ref{AC has no Pi-perfect subgps}, $g_i \in [N, \G \g_A]$ and so we may write
$$g_i = w_i(g_1, \ldots g_n) := \prod_k a_k [g_{j_k}, b_k]^{\pm}a_k^{-1} $$
where $a_k \in A$, $b_k \in \G \g_A$, and the product is finite.  Replace $g_i$ with $x_i$ to get a system of equations
$$ \left\{ x_i = w_i(x_1, \ldots x_n) \right\}_{i=1}^n $$
over $\g_A$ with distinct solution sets $\{x_i = e\}_{i=1}^n$ and $\{x_i = g_i\}_{i=1}^n$. \\

For the other half of the statement, suppose that 
$$\left\{ x_i = w_i(x_1, \ldots x_n) \right\}_{i=1}^n$$
is a system of equations over $\g_A$ with distinct solution sets $\{x_i = g_i \}_{i=1}^n$ and $\{x_i = h_i\}_{i=1}^n$.  Substitute $x_i = y_i h_i$ into these equations to obtain
$$\left\{ y_i = w_i(y_1 h_1, \ldots y_n h_n) h_i^{-1} =: \tilde{w}_i(y_1, \ldots y_n) \right\}_{i=1}^n .$$
Since $h_i \in \G \g_A$,
$$\tilde{w}_i \in \text{Ker}\left\{ A\ast F(y_1, \ldots y_n) \rightarrow \M \ast F(y_1, \ldots, y_n) \right\}$$
for all $i$.  We therefore have a new system of equations over $\g_A$ with distinct solution sets $\{y_i = e\}_{i=1}^n$ and $\{y_i = g_ih_i^{-1} \}_{i=1}^n$. \\

Let $N = \langle g_1h_1^{-1}, \ldots, g_nh_n^{-1} \rangle$, so $N$ is finitely normally generated in $A$ and nontrivial.  It will be shown that $N = [N, \G \g_A]$.  
Each $\tilde{w}_i$, having $\{y_i = e\}_{i=1}^n$ as a solution set, is an element of $\text{Ker}\left\{ A\ast F(y_1, \ldots y_n) \rightarrow A \right\}$ and may therefore be written as a finite product
$$\tilde{w}_i = \prod_k \beta_k y_{j_k}^{n_k} \beta_k^{-1} $$
where $\beta_k \in A$ and $n_k \in \Z$.  We have
\begin{align}
\tilde{w_{i}}(y_1, \ldots y_n) & = \prod_k \beta_k y_{j_k}^{n_k} \beta_k^{-1} \nonumber \\
& = \prod_k [\beta_k, y_{j_k}^{n_k}] y_{j_k}^{n_k} \nonumber \\
& = \big{(} \prod_k (y_{j_1}^{n_1} \cdots y_{j_{k-1}}^{n_{k-1}}) [\beta_k, y_{j_k}^{n_k}] (y_{j_1}^{n_1} \cdots y_{j_{k-1}}^{n_{k-1}})^{-1}\big{)} \prod_k y_{j_k}^{n_k} \nonumber
\end{align}
Moreover, $\prod_k y_{j_k}^{n_k} = e$ and each $\beta_k \in \G \g_A$ because 
$$\tilde{w}_i \in \text{Ker}\left\{ A\ast F(y_1, \ldots y_n) \to \M \ast F(y_1, \ldots, y_n) \right\}.$$  
Substituting $g_i h_i^{-1}$ in for $y_i$ we see that $ g_i h_i^{-1} \in [\G \g_A, N]$.
\end{proof}

$\phantom{}$ \\

\begin{prop}
\label{I is a Pi-perfect subgroup of A}
If $N_1, \ldots, N_n$ are $\Pi$-perfect subgroup of $\g_A$ then so is 
$$N = \langle N_1 \cup \cdots \cup N_n \rangle.$$
\end{prop}
\begin{proof}
It is clear that $[N, \G \g_A] \subset N$.  To see the other inclusion, consider an element $h \in N$.  We may write
$$ h = \prod_k a_k n_k a_k^{-1} $$
where $a_k \in A$, $n_k \in N_j$ for some $j$ dependent on $k$, and the product is finite.  As $N_j$ is $\Pi$-perfect we may write
$$ n_k = \prod_i b_{k_i} [m_{k_i}, g_{k_i}]^{\pm} b_{k_i}^{-1} $$
where $b_{k_i} \in A$, $m_{k_i} \in N_j$, $g_{k_i} \in \G \g_A$, and the product is finite.  Hence, we have
\begin{align}
h & = \prod_k a_k n_k a_k^{-1} \nonumber \\
& = \prod_k a_k \left( \prod_i b_{k_i} [m_{k_i}, g_{k_i}]^{\pm} b_{k_i}^{-1} \right) a_k^{-1} \nonumber \\
& = \prod_k \left( \prod_i a_k b_{k_i} [m_{k_i}, g_{k_i}]^{\pm} b_{k_i}^{-1} a_k^{-1} \right), \nonumber
\end{align}
and it follows that $h \in [N, \G \g_A]$.  Since $N$ is finitely normally generated in $A$ by the union of the finite normal generating sets of $N_1, \ldots, N_n$, $N$ is $\Pi$-perfect.
\end{proof}

$\phantom{}$ \\

Let $S$ be the set of all systems of equations over $\g_A$.  Construct $\tilde{A}$ by adjoining $F(x_1^{\alpha}, \ldots, x_{n_{\alpha}}^{\alpha})$ to $A$ for all systems $\alpha \in S$ and adding relations $x_i^{\alpha} = w_i^{\alpha}$.  That is, define
$$ \tilde{A} = \dfrac{ A \ast \left( \ast_{\alpha} F(x_1^{\alpha}, \ldots, x_{n_{\alpha}}^{\alpha}) \right)}{ \langle (x_i^{\alpha})^{-1} w_i^{\alpha} | \alpha \in S \rangle }.$$
Note that $\tilde{A}$ admits an epimorphism 
$$\xymatrix{ \g_{\tilde{A}}: \tilde{A} \ar[r] & G }$$ 
defined by 
\begin{equation*}
\left\{
\begin{array}{ll}
\g_{\tilde{A}}(a) = \g_A(a) & \text{for } a \in A  \\
\phantom{*} \\
\g_{\tilde{A}}(x_i^{\alpha}) = e & \text{for all } x_i^{\alpha} 
\end{array} \right.
\end{equation*}
$\phantom{M}$ \\
Let $I \lhd \tilde{A}$ be the normal subgroup generated by the union of all $\Pi$-perfect subgroups of $\tilde{A}$.  Every $\Pi$-perfect subgroup is contained in $\G \g_{\tilde{A}}$, so $I \unlhd \G \g_{\tilde{A}}$.  Define $\hat{A} = \dfrac{\tilde{A}}{I}$.  Since $I $ is contained in $\G \g_{\tilde{A}}$, $\g_{\tilde{A}}$ determines an epimorphism 
$$\g_{\hat{A}}: \hat{A} \to G .$$
It will be shown that $\g_{\hat{A}}$ is the algebraic closure of $\g_A$.  First we need some lemmas. \\

$\phantom{}$ \\

\begin{lem}
\label{AC has no Pi-perfect subgroups}
$\g_{\hat{A}}$ has no nontrivial $\Pi$-perfect subgroups.
\end{lem}
\begin{proof}
Suppose that $N = \langle g_1, \ldots, g_n \rangle = [N, \G \g_{\hat{A}}]$ is a $\Pi$-perfect subgroup.  Write $g_i$ as a finite product
$$ g_i = \prod_k a_k[g_{j_k}, h_k]^{\pm} a_k^{-1} $$
with $a_k \in \hat{A}$ and $h_k \in \G \g_{\hat{A}}$.  Lift $a_k$, $h_k$, and each $g_j$ to $\tilde{a}_k$, $\tilde{h}_k$, and $\tilde{g}_j$ in $\tilde{A}$, respectively, to obtain an equation
$$ \tilde{g}_i = \left( \prod_k \tilde{a}_k[\tilde{g}_{j_k}, \tilde{h}_k]^{\pm} \tilde{a}_k^{-1} \right) \tilde{\eta}_i, $$
for some $\tilde{\eta}_i \in I$.  Note that, by the construction of $\hat{A}$, $\tilde{h}_k \in \G \g_{\tilde{A}}$.  We may write $\tilde{\eta}_i$ as a finite product of elements of $I$, each of which is contained in a $\Pi$-perfect subgroup.  The union of these $\Pi$-perfect subgroups together with the normal subgroup generated by $\{ \tilde{g}_1, \ldots, \tilde{g}_n \}$ is contained in a $\Pi$-perfect subgroup by Proposition \ref{I is a Pi-perfect subgroup of A}.  Hence, $g_i = e$ in $\hat{A}$.
\end{proof}

$\phantom{}$ \\

\begin{lem}
\label{hat(A) is AC}
$\g_{\hat{A}}$ is algebraically closed.
\end{lem}
\begin{proof}
Let $\{ y_k = w_k(y_1, \ldots y_n) \}_{k=1}^n$ with 
$$w_k \in \text{Ker}\left\{ \hat{A} \ast F(y_1, \ldots, y_n) \to G \ast F(y_1, \ldots, y_n) \right\}$$
be a system of equations over $\g_{\hat{A}}$.  We first show that this system has a solution, and then that it is unique.  Each $w_k$ contains only a finite number of elements $\hat{a}_1, \ldots \hat{a}_{s_k} \in \hat{A}$.  Lift $\hat{a}_i$ to $\tilde{a}_i \in \tilde{A}$.  Each $\tilde{a}_i$ is a finite product of elements in $A$ and elements $\{ x_{i_j} \}$, and each $x_{i_j}$ is contained in the solution set to some system of equations $\{ x_r^{i_j} = v_r^{i_j} \}$ over $A$.  Combining all these systems of equations, we get a finite system of equations 
$$ \left\{ v_l \right\}_{l=1}^m $$
over $A$ that contains in its solution set all the elements of $\tilde{A}$ that appear in $\tilde{a}_1, \ldots, \tilde{a}_{s_k}$, for $k = 1, \ldots, n$, and are not in $A$.  We now form a new, even larger system of equations \\
\begin{equation*}
\left\{
\begin{array}{l}
y_k = w_k(y_1, \ldots, y_n, x_1, \ldots, x_m) \nonumber \\
x_l = v_l(x_1, \ldots, x_m) \nonumber
\end{array} \right.
\end{equation*} \\
with 
$$w_k, v_l \in \text{Ker}\left\{ A \ast F(y_1, \ldots, y_n, x_1, \ldots, x_m) \to G \ast F(y_1, \ldots, y_n, x_1, \ldots, x_m) \right\}.$$ 
Every system of equations over $\g_A$ has a solution in $\tilde{A}$, and therefore in $\hat{A}$.  Moreover, the solution to $\{w_k, v_l \}_{k, l}$ is unique by Theorem \ref{Pi-perfect iff nonunique solutions} since, by Lemma \ref{AC has no Pi-perfect subgroups}, $\g_{\hat{A}}$ has no nontrivial $\Pi$-perfect subgroups.
\end{proof}

$\phantom{}$ \\

\begin{lem}
\label{hat(A) is the AC}
If $f: \g_A \to \g_X$ and $\g_X$ is algebraically closed then there exists a unique morphism $\hat{f}: \g_{\hat{A}} \to \g_X$ making the following diagram commute.
$$
\xymatrix{
A \ar[rr] \ar[dr]_{f}& & \hat{A} \ar[dl]^{\hat{f}} \\
& X
}
$$
\end{lem}
\begin{proof}
We begin by showing existence.  Let $f: \g_A \to \g_X$ with $\g_X$ algebraically closed and let $\{ x_i = w_i \}_{i=1}^n$ be a system of equations over $\g_A$.  $f$ extends to a homomorphism
$$\xymatrix{ f': A \ast F(x_1, \ldots, x_n) \ar[r] & X \ast F(x_1, \ldots, x_n) }$$
over $G$ (where, in both cases, $x_i \mapsto e$ in $G$), and $\{ x_i = f'(w_i) \}_{i=1}^n$ is a system of equations over $X$.  As $\g_X$ is algebraically closed, there is a solution $\{ x_i = g_i \}_{i=1}^n$ in $X$ to this system.  $f$ therefore extends to
$$\xymatrix{ A \ast F(x_1, \ldots, x_n) \ar[r] & X }$$
by sending $x_i \mapsto g_i$.  In this way we see that $f$ extends to a homomorphism
$$\xymatrix{ \tilde{f}: \tilde{A} \ar[r] & X }$$
over $G$.  To see that $I$ is contained in the kernel of $\tilde{f}$, suppose that $N$ is a $\Pi$-perfect subgroup of $\g_{\tilde{A}}$ and let $\{h_i \}_{i=1}^m$ be a normal generating set for $N$.  We showed in the proof of Theorem \ref{Pi-perfect iff nonunique solutions} that there is a system of equations $\{x_i = w_i \}_{i=1}^m$ over $\g_{\tilde{A}}$ with distinct solutions $\{x_i = e\}_{i=1}^m$ and $\{x_i = h_i\}_{i=1}^m$.  
This gives rise to a system of equations over $\g_X$, $\{x_i = \tilde{f}(w_i) \}_{i=1}^m$, with solutions $\{x_i = e\}_{i=1}^m$ and $\{x_i = \tilde{f}(h_i) \}_{i=1}^m$.  Since $\g_X$ is algebraically closed, $\tilde{f}(h_i) = e$  for all $i$, implying that $N$ is in the kernel of $\tilde{f}$.  It follows that $I$, the normal subgroup of $\tilde{A}$ generated by the union of all $\Pi$-perfect subgroups of $\g_{\tilde{A}}$, is in the kernel of $\tilde{f}$.  Hence, $\tilde{f}$ factors through a homomorphism $\hat{f}: \hat{A} \to X$ over $G$. \\

We now show that $\hat{f}$ is unique.  Suppose that $\hat{p}$ and $\hat{q}$ are two homomorphisms making the following diagram commute:
$$
\xymatrix{
A \ar[rr] \ar[dr]_f & & \hat{A} \ar[dl]_{\hat{p}}^{\hat{q}} \\
& X & 
}
$$
We know that $\hat{p}$ and $\hat{q}$ agree on $A$, so consider $x \in \hat{A}$, where $x = x_j$ in some system of equations $\{x_i = w_i \}_{i=1}^n$ over $\g_A$.  Extend $f$ to 
$$f': A \ast F(x_1, \ldots, x_n) \to X \ast F(x_1, \ldots, x_n).$$ 
Then $\{ x_i = f'(w_i)   \}_{i=1}^n$ is a system of equations over $\g_X$ with solutions $\{ x_i = \hat{p}(x_i) \}_{i=1}^n$ and $\{ x_i = \hat{q}(x_i) \}_{i=1}^n$.  As $\g_X$ is algebraically closed, $\hat{p}(x_i) = \hat{q}(x_i)$ for all $i$ and, in particular, $\hat{p}(x) = \hat{q}(x)$.
\end{proof}
$\phantom{}$ \\

\begin{thrm}
\label{AC exists}
$\g_{\hat{A}}$ is the algebraic closure of $\g_A$.
\end{thrm} 
\begin{proof} 
By Lemma \ref{hat(A) is AC} $\g_{\hat{A}}$ is algebraically closed, and by Lemma \ref{hat(A) is the AC} $\g_{\hat{A}}$ is initial.
\end{proof} 
$\phantom{}$ \\

Recall from Definition \ref{Omega} that $\Omega^{\M}$ is the class of morphisms in $\mathcal{G}^{\M}$ satisfying the following properties:
\begin{enumerate}
\item $f: \g_A \rightarrow \g_B$ is a morphism with $A$ finitely generated and $B$  finitely presented,
\item $\G \g_{A}$ and $\G \g_{B}$ are finitely normally generated in $A$ and $B$, respectively,
\item $f$ induces a normal surjection $\G \g_{A} \rightarrow \G \g_{B}$, and 
\item $f_{i}: H_{i}(A; \Z[G]) \rightarrow H_{i}(B; \Z[G])$ is an isomorphism for $i=1$ and an epimorphism for $i = 2$. \\
\end{enumerate}

Recall also Definition \ref{local object}:  An object $\g_X \in \mathcal{G}^{\M}$ is {\em $\Omega^{\M}$-local} if for any $f: \g_A \to \g_B$ in $\Omega^{\M}$ and for any morphism $\g_A \to \g_X$, there exists a unique morphism $g: \g_B \to \g_X$ making the following diagram commute:
$$
\xymatrix{
A \ar[d] \ar[r]^f & B \ar[dl]^g \\
X 
}
$$ 
$\phantom{}$ \\

\begin{lem}
\label{kernels are finitely normally generated}
If $f: A \rightarrow B$ is an epimorphism of groups with $A$ finitely generated and $B$ finitely related then $B$ is finitely presented and $\text{Ker} \{f\}$ is finitely normally generated in $A$.
\end{lem}
\begin{proof} 
Let $p_A: F_A \to A$ be an epimorphism from a finitely generated free group onto $A$ and let $a_1, \ldots, a_n$ generate $F_A$.  Let $\langle \beta_i, i \in I ~|~ r_1, \ldots r_m \rangle$ be a presentation for $B$, with $I$ a possibly infinite index set.  If we write $f \circ p_A(a_1), \ldots, f \circ p_A(a_n)$ as words $v_1, \ldots, v_n$, respectively, in the $\beta_i$ then another presentation for $B$ is
$$\langle \alpha_1, \ldots, \alpha_n, \beta_i, i \in I ~|~ r_1, \ldots r_m, ~\alpha_1^{-1}v_1, \ldots, \alpha_n^{-1}v_n \rangle .$$
Since $f \circ p_A$ is an epimorphism we may write each $\beta_i$ as a word in $\alpha_1, \ldots, \alpha_n$ and replace each occurrence of $\beta_i$ in $r_1, \ldots r_m, \alpha_1^{-1}v_1, \ldots, \alpha_n^{-1}v_n$ with this word, giving a finite presentation 
$$\langle \alpha_1, \ldots, \alpha_n ~|~ r'_1, \ldots, r'_m, v'_1, \ldots, v'_n \rangle$$
for $B$.  Lift the $r'_j$ and $v'_i$ to words $s_j$ and $w_i$ in $F_A$, respectively.  Then the kernel of $f \circ p_A$ is normally generated by $s_1, \ldots, s_m, w_1, \ldots, w_n$.  Since the following diagram commutes,
$$
\xymatrix{
F_A \ar[rr]^{f \circ p_A} \ar[rd]_{p_A} & & B \\
& A \ar[ru]_{f}
}
$$
the kernel of $f$ is contained in the image under $p_A$ of the kernel of $f \circ p_A$.  That is,
$$\text{Ker}\{f\} \subset p_A( \text{Ker}\{f \circ p_A\} ).$$
It follows that the kernel of $f$ is finitely normally generated in $A$.
\end{proof}
$\phantom{}$ \\

\begin{thrm}
\label{AC = local}
$\g_X$ is algebraically closed if and only if it is $\Omega^{\M}$-local.
\end{thrm}
\begin{proof}
Suppose that $\g_X$ is $\Omega^{\M}$-local and let $\{ x_i = w_i \}_{i=1}^n$ be a system of equations over $\g_X$.  Let $a_1, \ldots, a_s \in X$, where $a_1, \ldots, a_r$ (for $r \leq s$) are the finitely-many elements that appear in these words and $a_{r+1}, \ldots, a_s$ are chosen such that $\{ \g_X(a_1), \ldots, \g_X(a_s) \}$ generates $\M$.  Let $f: F(y_1, \ldots, y_s) \to X$ by $f(y_i) = a_i$ and let $\tilde{w}_i$ be a lift of $w_i$ to an equation over $F(y_1, \ldots, y_s)$.  Then 
$$\g_X \circ f: F(y_1, \ldots, y_s) \to \M $$ 
is an epimorphism, and we can extend it to 
$$\dfrac{F(y_1, \ldots, y_s) \ast F(x_1, \ldots, x_n)}{\langle x_i \tilde{w}_i^{-1} | 1 \leq i \leq n \rangle} \to \M$$ 
by sending each $x_i$ to the identity in $\M$.
The following is clearly a pushout diagram in the category of groups:
$$
\xymatrix{
F(y_1, \ldots, y_s) \ar[r]^(.3){\iota} \ar[d]^f & \dfrac{F(y_1, \ldots, y_s) \ast F(x_1, \ldots, x_n)}{\langle x_i \tilde{w}_i^{-1} | 1 \leq i \leq n \rangle} = Q \ar[d] \\
X \ar[r] & \dfrac{X \ast F(x_1, \ldots, x_n)}{\langle x_i w_i^{-1} | 1 \leq i \leq n \rangle} = P.
}
$$ 
A simple Mayer-Vietoris sequence argument together with Lemma \ref{kernels are finitely normally generated} show that $\iota \in \Omega^{\M}$.
Since $\g_X$ is $\Omega^{\M}$-local, there exists a unique homomorphism $Q \to X$ making the top triangle commute.  This extends uniquely to a homomorphism $P \to X$ by the pushout property of $P$, implying that $\g_X \in AC_G$. \\

Suppose now that $X$ is algebraically closed.  Let $f: \g_A \to \g_B$ be in $\Omega^{\M}$ and let $\g_A \to \g_X$ be a morphism in $\mathcal{G}^{\M}$.  We have the following commutative diagram
$$
\xymatrix{
A \ar[d] \ar[r]^f \ar[ddrr]^{\g_A} & B \ar[ddr]^{\g_B} \\
X \ar[rrd]^{\g_X} \\
& & G
}
$$
and we want to show that there is a unique map $\hat{f}: B \to X$ making this commute.  Let $b_1, \ldots, b_n$ generate $B$.  As $\g_A$ and $\g_B$ are surjective, there are elements $a_1, \ldots a_n \in A$ such that $b_i f(a_i^{-1}) \in \G \g_B$ (recall that $\G \g_B = \text{Ker}\{\g_B\}$).  If $\tilde{b}_i = b_i f(a_i^{-1})$ then $B$ is generated by $f(A) \cup \{ \tilde{b}_i \}_{i=1}^n$ because $b_i = \tilde{b}_i f(a_i)$.  Recall that $f \in \Omega^G$, so $\G \g_B$ is normally generated in $B$ by $f(\G \g_A)$.  We may therefore express each $\tilde{b}_i$ as a finite product of the form
$$ \tilde{b}_i = \prod_k v_k f(h_k) v_k^{-1} $$
where $v_k$ is a word in $f(A) \cup \{ \tilde{b}_i \}_{i=1}^n$ and $h_k \in \G \g_A$.  Replacing each $\tilde{b}_j$ with $x_j$ in this product, we obtain an equation in $f(A) \ast F(x_1, \ldots, x_n)$.  Lifting each $v_k$ to $\bar{v}_k$ in $A\ast F(x_1, \ldots, x_n)$, we obtain an equation 
$$x_i = w_i(x_1, \ldots, x_n) := \prod_k \bar{v}_k h_k \bar{v}_k^{-1}$$
with $w_i \in \text{Ker}\left\{ A \ast F(x_1, \ldots, x_n) \to G \ast F(x_1, \ldots, x_n) \right\}$.  Consider the following diagram
$$
\xymatrix{
A \ar[d] \ar[r]^(.3){\tilde{f}} \ar@/^3pc/[rr]^f & \dfrac{A \ast F(x_1, \ldots, x_n)}{\langle x_i^{-1}w_i | i=1, \ldots, n \rangle} \ar[r]^(.75)q \ar@{..>}[dl]^p & B \phantom{a}, \\
X \\
}
$$
where $q(x_i) = \tilde{b}_i$, making $q$ an epimorphism, and $\tilde{f}$ is inclusion.  There is a unique solution $\{ x_i = g_i \}_{i=1}^n$ in $X$ to the image of $\{x_i = w_i \}_{i=1}^n$ in $X \ast F(x_1, \ldots, x_n)$ because $\g_X \in AC_G$.  Define $p: \dfrac{A \ast F(x_1, \ldots, x_n)}{\langle x_i^{-1}w_i | i=1, \ldots, n \rangle} \to X$ by $p(x_i) = g_i$.  To show that $p$ induces a homomorphism $\hat{f}:B \to X$ it is enough to show that $\text{Ker}\{q\} \subset \text{Ker}\{p\}$.  Since $f, \tilde{f} \in \Omega^G$, $q \in \Omega^G$ as well.  If $N = \text{Ker}\{q\}$ then $N$ is contained in $\G \g_P$, where $P = \dfrac{A \ast F(x_1, \ldots, x_n)}{\langle x_i^{-1}w_i | i=1, \ldots, n \rangle}$.  As $q$ is an epimorphism, it restricts to an epimorphism
$$q: \G \g_P \to \G \g_B,$$
and a simple diagram chase shows that
$$\xymatrix{ 0 \ar[r] & N \ar[r] & \G \g_P \ar[r]^q & \G \g_B \ar[r] & 0 }$$
is a short exact sequence.  Examining the associated five-term exact sequence
$$\xymatrix{ H_2( \G \g_P) \ar[r]^{\text{onto}} & H_2(\G \g_B) \ar[r] & H_1(N) \otimes_{\G \g_B} \Z \ar[r] & H_1( \G \g_P) \ar[r]^{\cong} & H_1(\G \g_B) }$$
we see that $\dfrac{N}{[N, \G \g_P]} = H_1(N) \otimes_{\G \g_B} \Z = 0$.  In particular, $N$ is $\Pi$-perfect.  This shows that $p$ extends to a homomorphism $\hat{f}:B \to X$.  As in the proof of Lemma \ref{hat(A) is the AC}, $\hat{f}$ must be unique for otherwise we could find a system of equations over $\g_X$ with distinct solutions, contradiction the algebraic closure of $\g_X$.
\end{proof}

$\phantom{}$ \\

Notice that if $f: \g_A \to \g_B$ then $f$ induces a homomorphism $\hat{f}: \g_{\hat{A}} \to \g_{\hat{B}}$.  For, $f$ induces a homomorphism $\tilde{f}: \tilde{A} \to \tilde{B}$ and, as $\g_{\tilde{B}}$ is algebraically closed, the composition
$$\xymatrix{ \tilde{A} \ar[r]^{\tilde{f}} & \tilde{B} \ar[r] & \hat{B} }$$
contains $I$ in its kernel.  Hence, we obtain $\hat{f}: \g_{\hat{A}} \to \g_{\hat{B}}$.

\begin{defn}
Define a functor $E: \mathcal{G}^{\M} \to \mathcal{G}^{\M}$ by $E(\g_A) = \g_{\hat{A}}$.  There is a natural transformation $p: id_{\mathcal{G}^{\M}} \to E$ by defining $p(\g_A): \g_A \to \g_{\hat{A}}$ to be the canonical homomorphism $A \to {\hat{A}}$.  The proof of Theorem \ref{M-localization} now follows from Theorem \ref{alg clos localized omega} below.
\end{defn}

$\phantom{}$ \\

\begin{thrm}
\label{alg clos localized omega}
If $f: \g_A \to \g_B$ is in $\Omega^G$ then $f$ induces an isomorphism $\hat{f}: \g_{\hat{A}} \to \g_{\hat{B}}$.
\end{thrm}
\begin{proof}
Consider the following commutative diagram
$$
\xymatrix{
A \ar[r]^f \ar[d]_{p_A} & B \ar[d]^{p_B} \\
\hat{A} \ar[r]^{\hat{f}} & \hat{B} \phantom{a},
}
$$
where $p_A = p(\g_A)$ and $p_B = p(\g_B)$.  By Theorem \ref{AC = local} there is a unique homomorphism $g: \g_B \to \g_{\hat{A}}$ such that
$$
\xymatrix{
A \ar[r]^f \ar[d]_{p_A} & B \ar[dl]^g \\
\hat{A} \ar[d]^{\hat{f}} \\
\hat{B}
}
$$ 
commutes.  Since $\hat{B}$ is algebraically closed and $f \in \Omega^G$, $p_B: \g_B \to \g_{\hat{B}}$ is the unique homomorphism making
$$
\xymatrix{
A \ar[r]^f \ar[d]_{p_A} & B \ar[d]^{p_B} \ar[dl]_g \\
\hat{A} \ar[r]^{\hat{f}} & \hat{B}
}
$$
commute.  $\g_{\hat{B}}$ is the algebraic closure of $\g_B$, so $g$ factors through a unique homomorphism $\hat{g}: \g_{\hat{B}} \to \g_{\hat{A}}$, as in the following commutative diagram.
$$
\xymatrix{
A \ar[r]^f \ar[d]_{p_A} & B \ar[d]^{p_B} \ar[dl]_g \\
\hat{A} \ar[r]^{\hat{f}} & \hat{B} \ar@/^/[l]^{\hat{g}}
}
$$
Since 
$$(\hat{g} \circ \hat{f}) \circ p_A = \hat{g} \circ \hat{f} \circ g \circ f = \hat{g} \circ p_B \circ f = g \circ f = p_A, $$ 
and since $id_{\hat{A}}: \hat{A} \to \hat{A}$ is the unique homomorphism such that $id_{\hat{A}} \circ p_A = p_A$, it follows that $\hat{g} \circ \hat{f} = id_{\hat{A}}$.  This shows that $\hat{f}$ is injective.  On the other hand, 
$$(\hat{f} \circ \hat{g}) \circ p_B = \hat{f} \circ g = p_B$$
and, as with $id_{\hat{A}}$, we see that $\hat{f} \circ \hat{g} = id_{\hat{B}}$.  This shows that $\hat{f}$ is surjective.
\end{proof}

\bibliography{BigBibliography}
\bibliographystyle{plain}
\end{document}